\documentclass{amsart}
\UseRawInputEncoding
\usepackage{amsmath,amssymb}
\usepackage{color,mathdots}
\usepackage{stmaryrd}
\usepackage{comment}
\usepackage{hyperref}
\usepackage{IEEEtrantools}
\usepackage{tikz}

\newtheorem{theorem}{Theorem}[section]
\newtheorem{lemma}[theorem]{Lemma}
\newtheorem{corollary}[theorem]{Corollary}
\newtheorem{fact}[theorem]{Fact}
\newtheorem{proposition}[theorem]{Proposition}

\theoremstyle{definition}
\newtheorem{example}[theorem]{Example}
\newtheorem{question}[theorem]{Question}
\newtheorem{remark}[theorem]{Remark}
\newtheorem{definition}[theorem]{Definition}

\def \dcl{\operatorname{dcl}}
\def \acl{\operatorname{acl}}
\def \Lin{\operatorname{Lin}}

\def \tp{\operatorname{tp}}
\def \id{\operatorname{id}}

\def \V {\mathcal V}

\def \aut {\operatorname{Aut}}

\def \theo {\operatorname{Th}}

\def \FC {\mathfrak{C}}
\def \CL {\mathcal{L}}
\def \Zz {\mathbb{Z}}

\def \dom{\operatorname{dom}}

\def \Ord{\operatorname{Ord}}
\def \Lim{\operatorname{Lim}}

\DeclareMathOperator{\scf}{SCF}
\DeclareMathOperator{\sep}{sep}

\DeclareMathOperator{\irk}{inp-rk}
\DeclareMathOperator{\drk}{dp-rk}

\def\Ind#1#2{#1\setbox0=\hbox{$#1x$}\kern\wd0\hbox to 0pt{\hss$#1\mid$\hss}
\lower.9\ht0\hbox to 0pt{\hss$#1\smile$\hss}\kern\wd0}
\def\ind{\mathop{\mathpalette\Ind{}}}
\def\Notind#1#2{#1\setbox0=\hbox{$#1x$}\kern\wd0\hbox to 0pt{\mathchardef
\nn=12854\hss$#1\nn$\kern1.4\wd0\hss}\hbox to
0pt{\hss$#1\mid$\hss}\lower.9\ht0 \hbox to
0pt{\hss$#1\smile$\hss}\kern\wd0}
\def\nind{\mathop{\mathpalette\Notind{}}}

\title{On rank not only in NSOP$_1$ theories}

\author{Jan Dobrowolski$^{\ast}$}
\thanks{$^{\ast}$Supported by DFG project BA
6785/2-1}
\address{Jan Dobrowolski, Institute for Mathematical Logic and Basic Research, 
Department of Mathematics and Computer Science at the University of Münster\\
Münster\\ Germany, {\em and}\newline Instytut Matematyczny\\
Uniwersystet Wroc\l{}awski\\
Wroc\l{}aw\\
Poland}
\email{dobrowol@math.uni.wroc.pl}

\author{Daniel Max Hoffmann$^{\dagger}$}
\thanks{$^{\dagger}$SDG}
\address{Daniel Max Hoffmann, Instytut Matematyki\\
Uniwersytet Warszawski\\
Warszawa\\
Poland}
\email{daniel.max.hoffmann@gmail.com}
\urladdr{{https://sites.google.com/site/danielmaxhoffmann/}}

\date{\today}

\begin{document}

\maketitle

\begin{abstract}
We introduce a family of local ranks $D_Q$ depending on a finite set $Q$ of pairs of the form $(\varphi(x,y),q(y))$ where $\varphi(x,y)$ is a formula and $q(y)$ is a global type. We prove that in any NSOP$_1$ theory these ranks satisfy some desirable properties; in particular,  $D_Q(x=x)<\omega$ for any finite variable $x$ and any $Q$, if $q\supseteq p$ is a Kim-forking extension of types, then $D_Q(q)<D_Q(p)$ for some $Q$, and if $q\supseteq p$ is a Kim-non-forking extension, then $D_Q(q)=D_Q(p)$ for every $Q$ that involves only invariant types whose Morley powers are $\ind^K$-stationary. We give natural examples of families of invariant types satisfying this property in some NSOP$_1$ theories.

We also answer a question of Granger about equivalence of dividing and dividing finitely in the theory $T_\infty$ of vector spaces with a generic bilinear form. We conclude that forking equals dividing in $T_\infty$, strengthening an earlier observation that $T_\infty$ satisfies the existence axiom for forking independence.

Finally, we slightly modify our definitions and go beyond NSOP$_1$ to find out that our local ranks are bounded by the well-known ranks: the inp-rank (\emph{burden}), and hence, in particular, by the dp-rank. Therefore, our local ranks are finite provided that the dp-rank is finite, for example if $T$ is dp-minimal. 
Hence, our notion of ranks identifies a non-trivial class of  theories containing all NSOP$_1$ and NTP$_2$ theories.
\end{abstract}

\section{Introduction}
In the past years, we observed a rapid development of geometric tools and techniques related to model-theoretic stability theory. After a successful use of these techniques in the context of stable theories and remarkable applications in algebraic geometry, studies went beyond the class of stable theories.
One of the main tools of a geometric nature in model theory is the notion of an independence relation (cf. \cite{HansThesis}), which plays a key role in the description of simple theories (cf. \cite{KimPi}). Another important geometric tool in model theory is the notion of a rank, which also can be used to characterize dividing lines in the stability hierarchy. For example a theory is simple if and only if the local rank $D(x=x,\varphi,k)$ is finite for every choice of a formula $\varphi$ and every natural number $k$ (cf. Proposition 3.13 in \cite{casasimpl}). Actually, in the case of simple theories there is an elegant connection between the local rank $D$ and forking independence $\ind$, in short: the rank decreases in an extension of types if and only if  this extension is a forking extension (Proposition 5.22 in \cite{casasimpl}).
On top of that, the local rank in simple theories was used to develop the theory of generics there (\cite{simpleGenerics}).

Independence relations and ranks behave less nicely in the case of non-simple NSOP$_1$ theories.
The NSOP$_1$ theories were defined in \cite{sheDza}, then studied more intensively in \cite{ArtemNick} and in \cite{kaplanramsey2017}, where also the ideas from Kim's talk (\cite{KimTalk}) came to the picture (roughly speaking: Kim proposed a notion of independence corresponding to non-dividing along Morley sequences).
Further studies on NSOP$_1$ focused on proving desired properties of the notion of independence related to the notion of Kim-forking as defined in \cite{kaplanramsey2017} (where Kim-dividing and Kim-forking were defined with the use of global invariant types), e.g. \cite{kaplanramsey2021} and \cite{KRS}.
The problem with this approach is that sometimes there are no invariant global types extending a given type
 (however, everything is fine if we work only over models).
Then, in \cite{DKR}, the authors redefined the notions of Kim-dividing, Kim-forking and Kim-independence to avoid this obstacle and worked with definitions more in the spirit of \cite{KimTalk}. However, they needed an extra assumption, i.e. they were working in NSOP$_1$ theories enjoying the existence axiom for  forking independence.
The NSOP$_1$ theories enjoying the existence axiom were also studied  in \cite{CKR}, where e.g. transitivity of Kim-independence (as defined in \cite{DKR}) was obtained over arbitrary sets.
The important question, whether every NSOP$_1$ theory automatically enjoys the existence axiom for forking independence remains open.

Similarly as forking independence in the case of simple theories, Kim-independen\-ce was used to describe the class of NSOP$_1$ theories (\cite{kaplanramsey2017}). Therefore one could expect that there should also exist a good notion of a rank, which, similarly to the situation in simple theories, is related to Kim-independence in the context of NSOP$_1$ theories and which  also describes the class of NSOP$_1$ theories (i.e. the rank is finite if and only if the theory is NSOP$_1$).
Some attempts to define such a rank for NSOP$_1$ theories were made in \cite{CKR}, however they were not fully successful in relating the rank to  Kim-independence (see Question 4.9 in \cite{CKR}).
On the other hand, the rank defined in \cite{CKR} is finite provided $T$ is NSOP$_1$ with existence
and, in a private communication, Byunghan Kim informed us that SOP$_1$ implies that this rank is not finite 
(for some formula $\varphi$, natural number $k$ and some type $q$).
Thus finiteness of the rank from \cite{CKR} characterizes the class of NSOP$_1$ theories.

Let us mention here that also the situation with generics in NSOP$_1$ groups is more difficult than in groups with simple theory. For example, the theory of vector spaces with a generic bilinear form with values in an algebraically closed field, does not have Kim-forking generics for the additive group of vector space (see \cite{bilinear}).
The theory of generics in NSOP$_1$ groups is currently under development and a suitable notion of rank could be very useful in that context.
%As we do not see a possible way to use the rank from %\cite{CKR} to define Kim-forking generics in the %NSOP$_1$ context,
%we decided to came up with our own notion of rank.

To summarize, for us, there were three main properties expected from the new notion of rank:
being finite if and only if the theory is NSOP$_1$,
being related to Kim-independence,
and having a prospective use in the development of generics in NSOP$_1$ groups.
Here is what we managed to obtain so far.
Our notion of rank (Definition \ref{def:Qrank}) is local and depends on pairs consisting of a formula and a global type.
It has all the usual properties of a rank and it is finite, provided the theory $T$ is NSOP$_1$.
Nevertheless, it is also finite outside the class of NSOP$_1$ theories (e.g. in DLO, see Example \ref{ex:Qrank.finite}), 
which was not expected, but makes the rank more interesting outside of the class of NSOP$_1$ theories (more on that in Section \ref{sec:beyond}).
To obtain a rank which is related to the notion of Kim-independence, we follow some ideas from the doctoral thesis of Hans Adler (see Section 2.4 in \cite{HansThesis}). More precisely, our rank is not a foundation rank (i.e. defined recursively), but a rank which is witnessed by our account on \emph{dividing patterns} (\cite{HansThesis}) related to Kim-dividing given as in \cite{kaplanramsey2017}.
In Section \ref{sec:rank.ind}, we explain the connection between our rank and Kim-independence, which depends on some notions of \emph{stationarity}.

After noticing that our notion of rank might be finite outside of NSOP$_1$, we investigated behaviour of the rank in the case of NTP$_2$ theories. It turned out that a slight modification of the main definition (compare Definition \ref{def:Qrank} and Definition \ref{def:Qrank2}) results in finiteness of the rank in any theory of finite dp-rank. Moreover, the modified rank is bounded by the inp-rank and hence by the dp-rank. On top of that, the aforementioned modification does not affect the notion of our rank in the case of NSOP$_1$ theories. %, as in this case the modifications can be made ``for free". 
Therefore, in this paper, we provide a notion of local rank which shares finiteness in two opposite corners of the stability hierarchy, which seems to be quite intriguing. Perhaps, a good notion of rank will be more  suitable to work across different dividing lines in the neo-stability hierarchy than a notion of independence (like, for example,  thorn-independence).
On top of that, to provide the definition of ranks in Section \ref{sec:beyond}, we introduce  \emph{semi-global types} and show Kim's lemma for Kim-dividing witnessed via sequences in semi-global types.

In Section \ref{sec:bilinear} we study forking in the theory $T_\infty$ of infinite-dimensional vector spaces over an algebraically closed field with a generic bilinear form, which is one of the main algebraic examples of a non-simple NSOP$_1$ theory.
We describe forking of formulae in $T_\infty$, answering in particular a question from \cite{Granger} about equivalence of dividing and dividing finitely.

The paper is organized as follows. In Section \ref{sec:basics}, we recall definitions and several facts needed later. Section \ref{sec:rank} contains the definition and some basic properties of the new rank; then the new rank is related to the notion of Kim-independence. 
Section \ref{sec:bilinear} is focused on forking in the theory $T_{\infty}$ and verifies auxiliary notions introduced in the previous section.
In Section \ref{sec:examples}, we collect more examples of NSOP$_1$ theories and we discuss the three relevant notions of independence.
Finally, in Section \ref{sec:beyond}, we go with our rank beyond the class of NSOP$_1$ theories.

We thank Zo\'{e} Chatzidakis, Itay Kaplan, Byughan Kim and Tomasz Rzepecki
for helpful discussions and comments.

\section{Basics about NSOP$_1$}\label{sec:basics}
As usual, we work with a complete $\CL$-theory $T$ and with a monster model $\FC$ of $T$, i.e. a $\kappa$-strongly homogeneous and $\kappa$-saturated model of $T$ for some big cardinal $\kappa$. By a small tuple/subset/substructure we mean some tuple/subset/substructure of size strictly smaller than $\kappa$. Unless stated otherwise, all considered tuples/subsets/substructures will be small. For example $A\subseteq\FC$ tacitly implies that $|A|<\kappa$.
In short, we follow conventions being standard in model theory, e.g. outlined in \cite{casasimpl} and \cite{tentzieg}.

At the beginning, we need to evoke several definitions and facts about Kim-dividing and NSOP$_1$ theories.
A reader unfamiliar with the subject may consult e.g. \cite{kaplanramsey2017}, \cite{DKR} and \cite{CKR}.
%We finish this with a well-known example of a NSOP$_1$ theory, which will be elaborated later on.

\begin{definition}
A formula $\varphi(x;y)$ has SOP$_1$ (\emph{Strict Order Property of the first kind}) if there exists a collection of tuples $(a_{\eta})_{\eta\in 2^{<\omega}}$ such that
\begin{enumerate}
    \item $\{\varphi(x;a_{\eta|_{\alpha}})\;|\;\alpha<\omega\}$ is consistent for every $\eta\in 2^\omega$,
    \item if $\eta\in 2^{<\omega}$ and $\nu \trianglerighteq\eta\smallfrown\langle 0\rangle$ then
    $\{\varphi(x;a_{\nu}),\,\varphi(x;a_{\eta\smallfrown\langle 1\rangle})\}$ is inconsistent,
\end{enumerate}
where $\trianglelefteq$ denotes the tree partial order on $2^{<\omega}$. The theory $T$ has \emph{SOP$_1$} if there is a formula which has SOP$_1$. Otherwise we say that $T$ is \emph{NSOP$_1$}.
\end{definition}

\begin{definition}
We say that $T$ \emph{enjoys the existence axiom for forking independence} if for each set $A$ and each tuple $b$ we have that $b\ind_A A$.
If $T$ is NSOP$_1$ and enjoys the existence axiom for forking independence, then we say that \emph{$T$ is NSOP$_1$ with existence}.
\end{definition}

\begin{definition}[Morley sequence in a type]
Let $A\subseteq\FC$, $p(y)\in S(A)$, and let $(I,<)$ be a linearly ordered set.
We call a sequence $(b_i)_{i\in I}$ \emph{a Morley sequence in $p$} if
\begin{enumerate}
    \item $b_i\ind_A b_{<i}$ for each $i\in I$,
    \item $(b_i)_{i\in I}$ is $A$-indiscernible,
    \item $b_i\models p$ for each $i\in I$.
\end{enumerate}
\end{definition}

\begin{definition}[Morley sequence in an invariant global type]\label{def:morley.global}
Assume that $q(y)\in S(\FC)$ (i.e. $q$ is a global type) is $A$-invariant and let $(I,<)$ be a linearly ordered set.
By a \emph{Morley sequence in $q$ over $A$ (of order type $I$)} 
we understand a sequence $\bar{b}=(b_i)_{i\in I}$ such that $b_i\models q|_{Ab_{<i}}$ for each $i\in I$.
By $q^{\otimes I}$ we indicate a global $A$-invariant type in variables $(x_i)_{i\in I}$
such that for any $B\supseteq A$, 
if $\bar{b}\models q^{\otimes I}|_B$ then then $b_i\models q|_{Bb_{<i}}$ for all $i\in I$.
\end{definition}
Therefore, if $q(y)\in S(\FC)$ is $A$-invariant, we have two possible notions of a Morley sequence:
a Morley sequence in $q|_A$ and a Morley sequence in $q$ over $A$. Of course, a Morley sequence in $q$ over $A$ is a Morley sequence in $q|_A$. The converse does not hold in general.

\begin{definition}
Let $q(y)$ be an $A$-invariant global type.
We say that a formula $\varphi(x;y)$ $q$-divides over $A$ if for some (equivalently: any)
Morley sequence $(b_i)_{i<\omega}$ in $q$ over $A$, the set $\{\varphi(x,b_i)\;|\;i<\omega\}$ is inconsistent.
\end{definition}

We have the following two notions of \emph{Kim-dividing}, (A) appears in \cite{DKR}
and (B) appear in \cite{kaplanramsey2017}.  
\textbf{For us notion (A) is the one which we will use here}.
However, Theorem 7.7 from \cite{kaplanramsey2017} implies that (A) and (B) coincide over a model (i.e. if $A=M\preceq\FC$) provided $T$ is NSOP$_1$,
and therefore we will switch very often to notion (B) if the situation is placed over a model.

\begin{definition}\label{def:Kim.div}
Let $A\subseteq\FC$.
\begin{enumerate}
\item[(A)]
Assume that $k\in\mathbb{N}\setminus\{0\}$.
We say that $\varphi(x,b)$ \emph{$k$-Kim-divides over $A$} if there exists
a Morley sequence $(b_i)_{i<\omega}$ in $\tp(b/A)$ such that
$\{\varphi(x,b_i)\;|\;i<\omega\}$ is $k$-inconsistent.
We say that $\varphi(x,b)$ \emph{Kim-divides over $A$} if there exists $k\in\mathbb{N}\setminus\{0\}$ such that $\varphi(x,b)$ $k$-Kim-divides over $A$.

\item[(B)]
We say that $\varphi(x;b)$ \emph{Kim-divides over $A$} if there exists an $A$-invariant global type $q(y)\supseteq\tp(b/A)$ such that $\varphi(x;y)$ $q$-divides over $A$.
Equivalently, we say that $\varphi(x;b)$ Kim-divides over $A$ if there exists an $A$-invariant global type $q(y)\supseteq\tp(b/A)$ and a Morley sequence $(b_i)_{i<\omega}\models q^{\otimes\omega}|_A$ such that $b_0=b$ and the set $\{\varphi(x,b_i)\;|\;i<\omega\}$ is inconsistent.
\end{enumerate}
\end{definition}

\begin{fact}[Kim's lemma over models]\label{fact:Kims.lemma}
Assume that $T$ has NSOP$_1$ and let $M\preceq\FC$. The following are equivalent:
\begin{enumerate}
\item $\varphi(x;b)$ Kim-divides over $M$,
\item for any $M$-invariant global type $q(y)\supseteq\tp(b/M)$ and any Morley sequence $(b_i)_{i<\omega}\models q^{\otimes\omega}|_M$ such that $b_0=b$ we have that the set $\{\varphi(x,b_i)\;|\;i<\omega\}$ is inconsistent.
\end{enumerate}
\end{fact}

Later on, we will define a local rank related to the notion of Kim-dividing.
Our rank will focus on Kim-dividing over models, so one could ask how much of the picture is lost if we restrict our attention only to Kim-dividing over models.
First, let us make an observation easily following from the definition: 
if $\varphi(x,b)$ Kim-divides over $A$ with respect to definition (B) and $A\subseteq B$
then there exists $c\equiv_A b$ such that $c\ind_A B$ and such that $\varphi(x,c)$ Kim-divides over $B$ with respect to definition (B). This means that passing to Kim-dividing over models is not so harmful if we decide to work with the definition (B). The following lemma shows the same for the notion of Kim-dividing from the definition (A). The proof is not a surprise in any meaning, but let us follow the argument for a little warm-up.

\begin{lemma}[any $T$]\label{lemma:base.change}
Let $A\subseteq B$, $\varphi(x;y)\in\CL(A)$.
If $\varphi(x,a)$ $k$-Kim-divides over $A$, then there exists $c\equiv_A a$ such that $c\ind_A B$ and
$\varphi(x,c)$ $k$-Kim-divides over $B$.
\end{lemma}

\begin{proof}
By the definition of $k$-Kim-dividing, there exists an $A$-indiscernible sequence $(a_i)_{i<\omega}$ such that $a_0=a$, for each $i<\omega$ we have $a_i\ind_A Aa_{<i}$ and $\{\varphi(x,a_i)\;|\;i<\omega\}$ is $k$-inconsistent.
We will use several properties of $\ind$, which hold in any theory $T$, these are listed e.g. in Remark 5.3 in \cite{casasimpl}.

\textbf{Step 1:} \textit{increase the length of $(a_i)_{i<\omega}$ to a big cardinal $\lambda$ (big enough for the use of Erd\H{o}s-Rado theorem for $B$-indiscernibility, see e.g. Propostion 1.6 in \cite{casasimpl})}. To do it, consider:
$$p(\bar{x})=p\big((x_{\alpha})_{\alpha<\lambda} \big):=\bigcup\limits_{\substack{
n<\omega \\ \alpha_0<\ldots<\alpha_n<\lambda}}\tp(a_0\ldots a_n/Aa_0)[x_{\alpha_0},\ldots,x_{\alpha_n}].$$
Let $\bar{b}=(b_{\alpha})_{\alpha<\lambda}\models p(\bar{x})$, then $\bar{b}$ is $A$-indiscernible, $b_0=a_0=a$,
$\{\varphi(x,b_{\alpha}\;|\;\alpha<\lambda\}$ is $k$-inconsistent. Invariance and finite character of $\ind$ imply that also $b_{\alpha}\ind_A Ab_{<\alpha}$ for each $\alpha<\lambda$.

\textbf{Step 2:} \textit{force $\ind$-independence over $B$}.
We define recursively partial elementary over $A$ maps $f_{\alpha}:\dcl(Ab_{\leqslant\alpha})\to\FC$, where $\alpha<\lambda$, such that
\begin{itemize}
    \item $f_{\alpha+1}|_{\dcl(Ab_{\leqslant\alpha})}=f_{\alpha}|_{\dcl{Ab_{\leqslant\alpha})}}$,
    \item $f_{\alpha}(b_{\alpha})\ind_A Bf_{\alpha}(b_{<\alpha})$,
\end{itemize}
for each $\alpha<\lambda$. We start with obtaining $f_0$.
Since $b_0\ind_A A$, $\tp(b_0/A)$ does not fork over $A$ and so there exists a non-forking extension:
$$\tp(b_0/A)\subseteq \tp(b_0'/B).$$
Let $f_0$ be determined by $f_0(b_0)=b_0'$ and $f_0|_A=\id_A$. Now, we deal with the successor step $\alpha\rightsquigarrow\alpha+1$.
Let $f_{\alpha}'\in\aut(\FC)$ extend $f_{\alpha}$.
Because $b_{\alpha+1}\ind_A Ab_{\leqslant\alpha}$, we have that $f'_{\alpha}(b_{\alpha+1})\ind_A Af_{\alpha}(b_{\leqslant\alpha})$.
There exists an extension
$$\tp\Big(f'_{\alpha}(b_{\alpha+1})/Af_{\alpha}(b_{\leqslant\alpha})\Big)\subseteq\tp\Big(b'_{\alpha+1}/Bf_{\alpha}(b_{\leqslant\alpha})\Big),$$
which does not fork over $A$, i.e. $b'_{\alpha+1}\ind_A Bf_{\alpha}(b_{\leqslant\alpha})$.
Let $f_{\alpha+1}$ be determined by $f_{\alpha+1}(b_{\alpha+1})=b_{\alpha+1}'$ and $f_{\alpha+1}|_{\dcl(Ab_{\leqslant\alpha})}=f_{\alpha}|_{\dcl(Ab_{\leqslant\alpha})}$.
We move on to the limit ordinal step: $\beta<\lambda$ and $\beta\in\Lim$.
Let $f^-_{\beta}\in\aut(\FC)$ extend $\bigcup_{\alpha<\beta}f_{\alpha}$. From $b_{\beta}\ind_A Ab_{<\beta}$ we obtain
$f^-_{\beta}(b_{\beta})\ind_A A(f_{\alpha}(b_{\alpha}))_{\alpha<\beta}$, so, again, we can find an extension
$$\tp\Big(f^-_{\beta}(b_{\beta})/A(f_{\alpha}(b_{\alpha}))_{\alpha<\beta}\Big)\subseteq\tp(b'_{\beta}/B(f_{\alpha}(b_{\alpha}))_{\alpha<\beta}\Big)$$
which does not fork over $A$, i.e. $b'_{\beta}\ind_A B(f_{\alpha}(b_{\alpha}))_{\alpha<\beta}$.
Let $f_{\beta}$ be determined by $f_{\beta}(b_{\beta})=b'_{\beta}$ and $f_{\beta}|_{\dcl(Ab_{<\beta})}=f^-_{\beta}|_{\dcl(Ab_{<\beta})}$.
After the recursion is done, we take $f\in\aut(\FC)$ which extends $\bigcup_{\alpha<\lambda}f_{\alpha}$ and set $(b'_{\alpha})_{\alpha<\lambda}=(f(b_{\alpha}))_{\alpha<\lambda}$.

\textbf{Step 3:} \textit{Erd\H{o}s-Rado}.
By Erd\H{o}s-Rado theorem (e.g. Proposition 1.6 in \cite{casasimpl}), there exists a $B$-indiscernible sequence $(c_i)_{i<\omega}$
such that for each $n<\omega$ there exist $\alpha_0<\ldots<\alpha_n<\lambda$ satisfying
$$c_0\ldots c_n\equiv_B b'_{\alpha_0}\ldots b'_{\alpha_n}.$$
Thus
\begin{itemize}
    \item $c_0\equiv_B b'_{\alpha_0}\equiv_A b'_0\equiv_A b_0=a$,
    \item $\{\varphi(x,c_i)\;|\;i<\omega\}$ is $k$-inconsistent (since $\bar{b}'\equiv_A\bar{b}$ and $\{\varphi(x,b_{\alpha})\;|\;\alpha<\lambda\}$ is $k$-inconsistent),
    \item $c_n\ind_B Bc_0\ldots c_{n-1}$ (because $c_0\ldots c_n\equiv_B b'_{\alpha_0}\ldots b'_{\alpha_n}$ and $b'_{\alpha_n}\ind_A Bb'_{<\alpha_n}$, and $\ind$ satisfies monotonicity, base-monotonicity and invariance).
\end{itemize}
Therefore $\varphi(x,c_0)$ $k$-Kim-divides over $B$ and $c_0\equiv_A a$. 

For some $\alpha_0<\lambda$ we have that $c_0\equiv_B b'_{\alpha_0}$. Then, because $b'_{\alpha_0}\ind_A Bb'_{<\alpha_0}$ and because $\ind$ satisfies monotonicity and invariance, we obtain that also $c_0\ind_A B$.
\end{proof}

\begin{definition}
\begin{enumerate}
    \item A partial type $\pi(x)$ \emph{Kim-forks} over $A$ if 
    $$\pi(x)\vdash\bigvee\limits_{j\leqslant n}\psi_j(x;b_j)$$
    for some $n<\omega$ and $\psi_j(x;b_j)$ Kim-divides over $A$ for each $j\leqslant n$.

    \item Let $a$ be a tuple from $\FC$ and let $A,B\subseteq\FC$. 
    We say that \emph{$a$ is Kim-independent from $B$ over $A$}, denoted by $a\ind^K_A B$, if
    $\tp(a/AB)$ does not Kim-fork over $A$.
\end{enumerate}
\end{definition}

\noindent
One could redefine the notions of \emph{Kim-forking} and \emph{Kim-independence} using Kim-dividing with respect to definition (B). In such a case, we will always indicate that we work with Kim-forking with respect to definition (B) or use $\ind^{K,q}$ to denote Kim-independence defined with Kim-dividing with respect to definition (B).

The previous lemma easily leads to the following.

\begin{corollary}\label{cor:easy}
Let $A\subseteq B$, $a,d\in\FC$.
\begin{enumerate}
    \item If $\varphi(x,a)$ Kim-divides over $A$, then there exists $c\equiv_A a$ such that $c\ind_A B$ and $\varphi(x,c)$ Kim-divides over $B$.
    
    \item If for all $c\equiv_A a$ such that $c\ind_A B$, the formula $\varphi(x,c)$ does not Kim-divide over $B$, then
    $\varphi(x,a)$ does not Kim-divide over $A$.
    
    \item 
    If for all $d'\equiv_A d$ and all $c\equiv_A a$ such that $c\ind_A B$ we have $d'\ind^K_B c$ then $d\ind^K_A a$.
\end{enumerate}
\end{corollary}

The following characterization of NSOP$_1$ proved much more useful in studying Kim-independence than the original definition of NSOP$_1$ introduced in \cite{sheDza}.

\begin{theorem}[Theorem 8.1 in \cite{kaplanramsey2017}]\label{Kim.lemma.model}
The following are equivalent.
\begin{enumerate}
    \item $T$ is NSOP$_1$.
    
    \item \emph{Kim's lemma for Kim-dividing}:
    For any $M\preceq\FC$ and any $\varphi(x;b)$, if $\varphi(x;y)$ $q$-divides over $M$ for \textbf{some} $M$-invariant $q(y)\in S(\FC)$ with $\tp(b/M)\subseteq q(y)$, then 
    $\varphi(x;y)$ $q$-divides over $M$ for \textbf{any} $M$-invariant $q(y)\in S(\FC)$ with $\tp(b/M)\subseteq q(y)$.
\end{enumerate}
\end{theorem}

\noindent
Fact \ref{fact:Kims.lemma} has a generalization to Kim-dividing over arbitrary sets:

\begin{theorem}[Theorem 3.5 in \cite{DKR}]\label{Kim.lemma.arbitrary}
Let $T$ be NSOP$_1$ with existence.
Then $T$ satisfies \emph{Kim's lemma for Kim-dividing over arbitrary sets}:
if a formula $\varphi(x,b)$ Kim-divides over $A$ with respect to \textbf{some} Morley sequence in $\tp(b/A)$ then
the formula $\varphi(x,b)$ Kim-divides over $A$ with respect to \textbf{any} Morley sequence in $\tp(b/A)$.
\end{theorem}

As we will notice in a moment, Kim's lemma for Kim-dividing (over models) will be the main reason behind the fact
that our local rank is finite in the context of NSOP$_1$ theories.

\section{Rank}\label{sec:rank}
\subsection{Definition and basic properties}
In this section, we are interested in defining a local rank depending on 
pairs consisting of an $\CL$-formula and a global type. We will prove several properties of this new rank.
Our idea for the rank was in some way motivated by Hans Adler's doctoral dissertation (\cite{HansThesis}). More precisely, in Section 2 of his dissertation, Adler defines so called \emph{dividing patterns} and then defines a local rank measuring the length of a maximal dividing pattern. In our case, we could not simply reuse his idea, since we are trying to ``domesticate" Kim-dividing, and instead of that we propose our own variation on \emph{Kim-dividing patterns}.

Let $Q:=\big((\varphi_0(x;y_0),q_0(y_0)),\ldots,(\varphi_{n-1}(x;y_{n-1}),q_{n-1}(y_{n-1}) )\big)$,
where $\varphi_0,\ldots,\varphi_{n-1}\in\CL$ and $q_0,\ldots,q_{n-1}$ are global types.

\begin{definition}\label{def:Qrank}
We define a local rank, called \emph{Q-rank}, 
$$D_Q(\,\cdot\,):\big\{\text{sets of formulae}\}\to\Ord\cup\{\infty\}.$$
For any set of $\mathcal{L}$-formulae $\pi(x)$
we have $D_Q(\pi(x))\geqslant\lambda$ if and only if there exists 
$\eta\in n^{\lambda}$ and 
$(b^{\alpha},M^{\alpha})_{\alpha<\lambda}$ such that
\begin{enumerate}
    \item $\dom(\pi(x))\subseteq M^0$,
    \item $q_0,\ldots,q_{n-1}$ are $M^0$-invariant,
    \item $M^{\alpha}\preceq\FC$ for each $\alpha<\lambda$
     and $(M^{\alpha})_{\alpha<\lambda}$ is continuous,
    \item $b^{\alpha}M^{\alpha}\subseteq M^{\alpha+1}$ for each $\alpha+1<\lambda$,
    \item $b^{\alpha}\models q_{\eta(\alpha)}|_{M^{\alpha}}$ for each $\alpha<\lambda$,    
    \item $\pi(x)\cup\{\varphi_{\eta(\alpha)}(x;b^{\alpha})\;|\;\alpha<\lambda\}$ is consistent,

    \item each $\varphi_{\eta(\alpha)}(x;b^{\alpha})$ Kim-divides over $M^{\alpha}$ with respect to definition (B), i.e.: for each $\alpha<\lambda$ there exists an $M^{\alpha}$-invariant global type $r_{\alpha}(y_{\eta(\alpha)})$ extending $\tp(b^{\alpha}/M^{\alpha})$ and $\bar{b}^{\alpha}=(b^{\alpha}_i)_{i<\omega}\models r_{\alpha}^{\otimes\omega}|_{M^{\alpha}}$ such that $b^{\alpha}_0=b^{\alpha}$ and $\{\varphi_{\eta(\alpha)}(x;b^{\alpha}_i)\;|\;i<\omega\}$ is inconsistent.
\end{enumerate}
If $D_Q(\pi)\geqslant\lambda$ for each $\lambda\in\Ord$,
then we set $D_Q(\pi)=\infty$. 
Otherwise $D_Q(\pi)$ is the maximal $\lambda\in\Ord$ such that $D_Q(\pi)\geqslant\lambda$.
\end{definition}

\begin{remark}\label{rem:cond.3}
If $T$ is NSOP$_1$, then Definition \ref{def:Qrank} is equivalent to the same definition but with the condition (3) replaced by
\begin{enumerate}
    \item[(3$^\ast$)] $M^{\alpha}\preceq\FC$ for each $\alpha<\lambda$, $(M^{\alpha})_{\alpha<\lambda}$ is continuous, and each $M^{\alpha+1}$ is $|M^{\alpha}|^+$-saturated,
\end{enumerate}
or even by
\begin{enumerate}
    \item[(3$^{\ast\ast}$)] $M^{\alpha}\preceq\FC$ for each $\alpha<\lambda$, $(M^{\alpha})_{\alpha<\lambda}$ is continuous, and each $M^{\alpha+1}$ is $|M^{\alpha}|^+$-saturated and strongly $|M^{\alpha}|^+$-homogeneous.
\end{enumerate}
The proof is a quite standard and long recursion, which uses transitivity and symmetry of $\ind^K$ over models.
Thus we omit it. 
If we refer to Definition \ref{def:Qrank} in the case of $T$ being NSOP$_1$, we usually have in mind its equivalent formulation with the condition (3$^\ast$).
\end{remark}

Let us explain a little bit the concept behind this rank. For simplicity we assume that $Q=( (\varphi,q) )$.
Then the witnesses from the definition of $D_Q(\pi)\geqslant\lambda$, $(M^{\alpha},b^{\alpha}_i)_{\alpha<\lambda,i<\omega}$, may be used to draw the following tree:
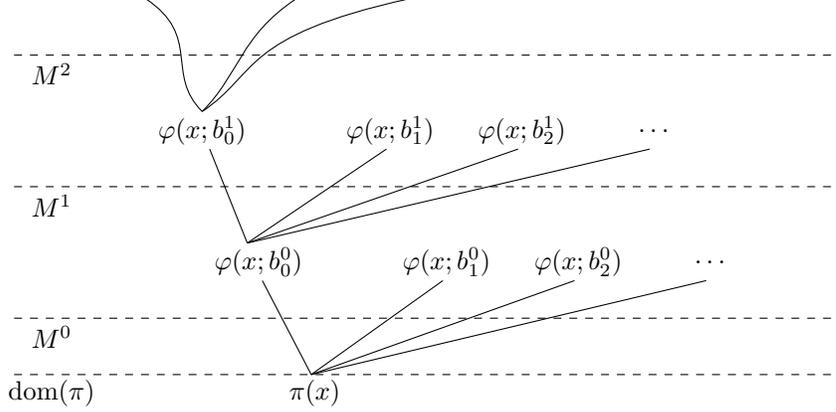
\begin{center}
\hspace{10mm}
\begin{tikzpicture}
\node (a) at (0,0) {$\pi(x)$};
\node (a) at (-0.75,1.75) {$\varphi(x;b^0_0)$};
\node (a) at (-1.5,3.5) {$\varphi(x;b^1_0)$};

\draw  (-0.05,0.25) -- (-0.7,1.5);
\draw  (-0.9,2) -- (-1.4,3.25);

\node (a) at (1.75,1.75) {$\varphi(x;b^0_1)$};
\node (a) at (3.5,1.75) {$\varphi(x;b^0_2)$};
\node (a) at (5.25,1.75) {$\ldots$};

\node (a) at (1,3.5) {$\varphi(x;b^1_1)$};
\node (a) at (2.75,3.5) {$\varphi(x;b^1_2)$};
\node (a) at (4.5,3.5) {$\ldots$};

\draw  (-0.05,0.25) -- (1.7,1.5);
\draw  (-0.05,0.25) -- (3.45,1.5);
\draw  (-0.05,0.25) -- (5.2,1.5);

\draw  (-0.9,2) -- (0.95,3.25);
\draw  (-0.9,2) -- (2.7,3.25);
\draw  (-0.9,2) -- (4.45,3.25);

\draw  (-1.5,3.75) .. controls (-2,4.25) and (-1.5,4.75) .. (-2.25,5.25);
\draw  (-1.5,3.75) .. controls (-1,4.25) and (-1,4.75) .. (-0.25,5.25);
\draw  (-1.5,3.75) .. controls (-0.75,4.25) and (-1,4.75) .. (1.25,5.25);

\draw[dashed] (-4,0.25) -- (7,0.25);
\node (a) at (-3.5,0) {$\dom(\pi)$};

\draw[dashed] (-4,1) -- (7,1);
\node (a) at (-3.5,0.75) {$M^0$};

\draw[dashed] (-4,2.75) -- (7,2.75);
\node (a) at (-3.5,2.5) {$M^1$};

\draw[dashed] (-4,4.5) -- (7,4.5);
\node (a) at (-3.5,4.25) {$M^2$};
\end{tikzpicture}
\end{center}
Each horizontal sequence of $b^{\alpha}_i$'s is a Morley sequence in some global $M^{\alpha}$-invariant type and so witnesses Kim-dividing of $\varphi(x;b^{\alpha}_0)$ over $M^{\alpha}$. 
This is nothing new.

The first new ingredient in our rank is that we require that also the leftmost branch in our tree forms a Morley sequence, this time in the previously fixed global type $q$ (which is $M^0$-invariant).
In other words, we focus only on Morley sequences - and that is in accordance with the intuition that all the essential data in a NSOP$_1$ theory is coded by Morley sequences.

The second new ingredient in our rank is that we allow ``jumps" in the extension of the base parameters between levels. More precisely, instead of the sequence $\dom(\pi)\subseteq M^0\preceq M^1\preceq\ldots$, we could consider a more standard sequence $\dom(\pi)\subseteq \dom(\pi)b^0_0\subseteq\dom(\pi)b^0_0b^1_0\subseteq\ldots$.
However, let us recall that $\ind^K$ does not necessarily satisfy the base monotonicity axiom, thus we allow in our rank some freedom in choosing the parameters over which each next level Kim-divides.

\begin{remark}
Because in NSOP$_1$ theories the both notions of Kim-dividing ( (A) and (B) from Definition \ref{def:Kim.div})
coincide over a model and in our rank we consider only Kim-dividing over models, our rank is suitable to work with both notions of Kim-dividing in the NSOP$_1$ environment.
\end{remark}

In the following lines, we will make use of our intuition and show several nice properties of the $Q$-rank,
we also refine the definition of $Q$-rank, so it will become more technical, but it will allow us to proof a few more facts. We start with something completely trivial.

\begin{fact}
\begin{enumerate}
\item 
$D_Q(\pi))\geqslant\lambda$, $f\in\aut(\FC)$ $\Rightarrow$ $D_{f(Q)}(f(\pi))\geqslant\lambda$.

\item 
If $\pi'\subseteq \pi$ then $D_Q(\pi)\leqslant D_{Q}(\pi')$.

\item
$D_Q(\pi)\leqslant\sum\limits_{j<n}D_{((\varphi_j,q_j))}(\pi)$.

\end{enumerate}
\end{fact}

\begin{proof}
All the three items follow by the definition.
\end{proof}

\begin{corollary}\label{cor:finite.single}
There exists finite $Q$ such that $D_Q(x=x)\geqslant\omega$ if and only if there exists $Q$ such that
$|Q|=1$ and $D_Q(x=x)\geqslant\omega$.
\end{corollary}

As we will see in a moment, the $Q$-rank is finite in the case of NSOP$_1$ theories,
which is a desired property of our rank.
It also happens that outside of NSOP$_1$ the rank may be finite, e.g. in the case of $T$ being DLO.
In the following example we work with Definition \ref{def:Qrank}, however a more general result on finiteness of our rank in the context of NIP theories is provided in Section \ref{sec:beyond}, where we work with a slightly modified definition of Q-rank (see Definition \ref{def:Qrank2}).

\begin{example}\label{ex:Qrank.finite}
Let $T$ be the theory of dense linear orders without endpoints, DLO.
First, we will show that $D_Q(x=x)\leq 1$ for  $Q=(\phi(x,yz),q)$, where $q$ is an arbitrary invariant type and  $ \phi(x,yz)=(y<x<z)$. Suppose $M^i, b^i_j, i< 2$, $j<\omega$ with  $b^i_j=(a^i_j,c^i_j)$ witness that $D_Q(x=x)\geq 2$.
As $\phi(x;a^1_0,c^1_0)$ divides over  $a^0_0c^0_0$ (i.e. over  $b^0_0$),   $a^0_0,c^0_0$ cannot lie in  $(a^1_0,c^1_0)$, hence by the consistency condition we must have $a^0_0<a^1_0<c^1_0<c^0_0$ (*). But then an automorphism over  $M_0$ moving   $b^0_0$ to $b^0_1$ will move  $\phi(x;a^1_0,c^1_0)$ to a formula inconsistent with it. This contradicts $M_0$-invariance of the restriction $q'$ of $q$ to the first variable as  $a^1_0\models q'|M_1$. 

Note also that if $\phi(x,y)$ is of the form $x>y$ or $x<y$ or $x=x$, then $D_{(\phi,q)}(x=x)=0$, as in that case no instance of $\phi(x,y)$ Kim-divides over any set. Also, it is easy to see that if $\phi(x,y)=(x=y)$ then $D_{(\phi,q)}(x=x)=1$.
	
Now let $Q=(\phi(x,y),q(y))$ with $x=x_0\dots x_{n-1}$ be a variable of length $n$, $\phi(x,y)$ a quantifier-free formula and $q(y)$ a global invariant type. By quantifier elimination and completeness of $q(y)$ we have that $q(y)\wedge \phi(x,y)\equiv q(y)\wedge \bigvee_{i<k} \phi_i(x,y)$ where each $\phi_i(x,y)$ defines a product of intervals with ends in the set $y\cup\{+\infty, -\infty \}$. We claim that $D_Q(x=x)\leq nk$. Suppose for a contradiction that $D_Q(x=x)> nk$.  Then by pigeonhole principle one easily gets that $D_{Q'}(x=x)> n$ where $Q'=(\phi_i,q)$ for some $i<k$. 

Let  $(b^i,M^i)_{i<n+1}$ witness that $D_{Q'}(x=x)\geq n+1$. Choose  $c\models \{\phi(x,b^i):i<N\}$ and let $\psi(x)$ be a formula equivalent to $qftp(c/\emptyset)$.
Wlog $\phi(x,y)\vdash \psi(x)$. Hence $\phi(x,y)$ is of the form $\psi(x)\wedge \bigwedge_{j<n} \chi_j(x_j,y)$, where each $\chi_j(x_j,y)$ defines an interval with ends in $y\cup\{+\infty, -\infty\}$.
Again by pigeonhole principle we must have that $D_{Q''}(x=x)\geq 2$ with  $Q''=(\psi(x)\wedge \chi_j(x_j,y)))$ for some $j<n$. As $\psi(x)$ is over $\emptyset$ and is consistent, clearly we get that any family of instances of $\psi(x)\wedge \chi_j(x_j,y))$ is consistent if and only if the corresponding family of instances of $\chi_j(x_j,y))$ is consistent. Hence $D_{Q''}(x=x)=D_{(\chi_j(x_j,y),q)}(x_j=x_j)$. But $D_{(\chi_j(x_j,y),q)}(x_j=x_j)\leq 1$ by the first two paragraphs, a contradiction.

\end{example}

\begin{definition}
Let $\varphi(x;y)\in\CL$ and let $q(y)\in S(\FC)$.
We define $C_{\varphi,q}\leqslant\omega$ as
\begin{IEEEeqnarray*}{rCl}
C_{\varphi,q} &:=& \max\{k\;|\;(\exists M\preceq\FC,\,M\text{ is $M$-invariant}) \\
& & (\exists (a_i)_{i<\omega}\models q^{\otimes\omega}|_M)(\{\varphi(x;a_i)\;|\;i<k\}\text{ is consistent})\}.
\end{IEEEeqnarray*}
\end{definition}

\begin{remark}
Let $M,N\preceq\FC$. Assume that
\begin{itemize}
\item $q$ is $M$-invariant, $(a_i)_{i<\omega}\models q^{\otimes\omega}|_M$ and
$k_M$ is the maximal number (or $\omega$) such that $\{\varphi(x;a_i)\;|\;i<k_M\}$ is consistent,

\item $q$ is $N$-invariant, $(b_i)_{i<\omega}\models q^{\otimes\omega}|_N$ and
$k_N$ is the maximal number (or $\omega$) such that $\{\varphi(x;b_i)\;|\;i<k_M\}$ is consistent.
\end{itemize}
Then $k_N=k_M$. To see this we introduce auxiliary $\bar{N}\preceq\FC$ which contains $M$ and $N$, and a Morley sequence $(c_i)_{i<\omega}\models q^{\otimes\omega}|_{\bar{N}}$.
We have that $(a_i)_{i<\omega}\equiv_M (c_i)_{i<\omega}\equiv_N (b_i)_{i<\omega}$.
This remark says that $C_{\varphi,q}$ is in some sense a uniform bound and it can not happen that
for each $M\preceq\FC$ and each corresponding Morley sequence $(a_i)_{i<\omega}$,
 $\max\{k\;|\;\{\varphi(x;a_i)\;|\;i<k\}\}$ is finite, but $C_{\varphi,q}=\omega$.
\end{remark}
\noindent
From now on (if not stated otherwise), \textbf{we assume that $T$ has NSOP$_1$.}

\begin{lemma}\label{lemma:Qrank.NSOP1}
Then $D_Q(\pi)\geqslant\lambda$ if and only if there exists
$\eta\in n^\lambda$ and $\big((b^{\alpha}_i)_{i<\omega},M^{\alpha}\big)_{\alpha<\lambda}$
such that
\begin{enumerate}
    \item $\dom(\pi(x))\subseteq M^0$,
    \item $q_0,\ldots,q_{n-1}$ are $M^0$-invariant,
    \item $M^{\alpha}\preceq\FC$ for each $\alpha<\lambda$, $(M^{\alpha})_{\alpha<\lambda}$ is continuous, and each $M^{\alpha+1}$ is $|M^{\alpha}|^+$-saturated,
    %and strongly $|M^{\alpha}|^+$-homogeneous,
    \item $(b^{\alpha}_i)_{i<\omega} M^{\alpha}\subseteq M^{\alpha+1}$ for each $\alpha+1<\lambda$,   
    \item $\pi(x)\cup\{\varphi_{\eta(\alpha)}(x;b^{\alpha}_0)\;|\;\alpha<\lambda\}$ is consistent,
 
 	\item for each $\alpha<\lambda$, we have $(b^{\alpha}_i)_{i<\omega}\models q^{\otimes\omega}_{\eta(\alpha)}|_{M^{\alpha}}$ and $\{\varphi_{\eta(\alpha)}(x;b^{\alpha}_i)\;|\;i<\omega\}$ is inconsistent.      
\end{enumerate}
\end{lemma}

\begin{proof}
The implication right-to-left is straightforward and holds even without the assumption about NSOP$_1$. Let us show the implication left-to-right.

By Kim's lemma (Fact \ref{fact:Kims.lemma}), we can replace the condition (7) from Definition \ref{def:Qrank}, by: for each $\alpha<\lambda$ there exists $\bar{c}^{\alpha}=(c^{\alpha}_i)_{i<\omega}\models q^{\otimes\omega}_{\eta(\alpha)}|_{M^{\alpha}}$ such that $c^{\alpha}_0=b^{\alpha}$ and $\{\varphi_{\eta(\alpha)}(x;c^{\alpha}_i)\;|\;i<\omega\}$ is inconsistent.
By saturation of $M^{\alpha+1}$, we find $(b^{\alpha}_i)_{i<\omega}\subseteq M^{\alpha+1}$ such that $(b^{\alpha}_i)_{i<\omega}\equiv_{M^{\alpha}b^{\alpha}}(c^{\alpha}_i)_{i<\omega}$ which is the desired sequence to finish the proof.
\end{proof}

\begin{lemma}\label{lemma:DQ.finite}
$D_Q(\{x=x\})\leqslant\sum\limits_{j\in J}C_{\varphi_j,q_j}<\omega$,
where $J=\{j< n\;|\;C_{\varphi_j,q_j}<\omega\}$.
\end{lemma}

\begin{proof}
Let us deal first with the case when $Q=((\varphi(x;y),q(y)))$.
Assume that $D_Q(\pi)\geqslant\lambda>0$ and let 
$(b^{\alpha},M^{\alpha})_{\alpha<\lambda}$ be as in Definition \ref{def:Qrank}
($\eta$ is constant, so we skip it).

We have that $b^{\alpha}\models q|_{M^{\alpha}}$ and since $b^{\alpha}M^{\alpha}\subseteq M^{\alpha+1}$, it is $b^{\alpha}\models q|_{M^0b^{<\alpha}}$. We know also that $q$ is $M^0$-invariant and that $\{\varphi(x;b^{\alpha})\;|\;\alpha<\lambda\}$ is consistent.
On the other hand, by Lemma \ref{lemma:Qrank.NSOP1},
there exists a Morley sequence $(b^0_i)_{i<\omega}\models q^{\otimes\omega}|_{M^0}$ such that
$b^0_0=b^0$ and $\{\varphi(x;b^0_i)\;|\;i<\omega\}$ is inconsistent.
Therefore $C_{\varphi,q}<\omega$ and so $\lambda\leqslant C_{\varphi,q}<\omega$.

We switch to the general case where
$$Q=\big((\varphi_0(x;y_0),q_0(y_0)),\ldots,(\varphi_{n-1}(x;y_{n-1}),q_{n-1}(y_{n-1}) \big).$$
Now, the function $\eta\in n^{\lambda}$ becomes important.
We have that $\lambda=\bigcup_{j< n}\eta^{-1}[j]$, so we are done if we can show that $|\eta^{-1}[j]|=0$ or $|\eta^{-1}[j]|\leqslant C_{\varphi_j,k_j}<\omega$ for each $j< n$.
Fix some $j< n$ and assume that $|\eta^{-1}[j]|>0$.
Let $\alpha<\lambda$ be the first index such that $\eta(\alpha)=j$.
Repeating the first part of this proof for $q_j$ we obtain what we need.
\end{proof}

\begin{corollary}\label{cor:DQ.finite}
$D_Q(\pi)\leqslant\sum\limits_{j\in J}C_{\varphi_j,q_j}<\omega$,
where $J=\{j< n\;|\;C_{\varphi_j,q_j}<\omega\}$.
\end{corollary}

\begin{lemma}\label{lemma:technical}
$D_Q(\pi)\geqslant\lambda$ if and only there exist $\eta\in n^{\lambda}$,
$(b^{\alpha}_i)_{\alpha<\lambda,i<\omega}$ and $M\preceq\FC$ such that
\begin{enumerate}
    \item $\dom(\pi)\subseteq M$,
    \item $q_0,\ldots,q_{n-1}$ are $M$-invariant,
    \item $\pi(x)\cup\{\varphi_{\eta(\alpha)}(x;b^{\alpha}_0)\;|\;\alpha<\lambda\}$ is consistent,
    \item $\{\varphi_{\eta(\alpha)}(x;b^{\alpha}_i)\;|\;i<\omega\}$ is inconsistent for each $\alpha<\lambda$,
    \item $$(b^{\lambda-1}_i)_{i<\omega}^\smallfrown\ldots^\smallfrown(b^0_i)_{i<\omega}\models q_{\eta(\lambda-1)}^{\otimes\omega}\otimes\ldots\otimes q_{\eta(0)}^{\otimes\omega}|_{M}.$$
\end{enumerate}
\end{lemma}

\begin{proof}
The left-to-right implication is straightforward: 
by Lemma \ref{lemma:Qrank.NSOP1} there are proper
$\eta\in n^\lambda$ and $\big((b^{\alpha}_i)_{i<\omega},M^{\alpha}\big)_{\alpha<\lambda}$, 
we set $M:=M^0$ and reuse $(b^{\alpha}_i)_{\alpha<\lambda,i<\omega}$.

Let us take care of the right-to-left implication.
We will recursively define a sequence 
$((e^{\alpha}_i)_{i<\omega},M^{\alpha})_{\alpha<\lambda}$
satisfying all the six conditions from Lemma \ref{lemma:Qrank.NSOP1}.
Because we do not want to get lost in a notational madness, 
instead of introducing new subscripts,
we will sketch a few first steps.

We set $M^0:=M$ and take $M^1\preceq\FC$ which is a $|M^0|^+$-saturated, and which contains $M^0$ and $(b^0_i)_{i<\omega}$.
Consider a Morley sequence
$$(c^{\lambda-1}_i)_{i<\omega}^\smallfrown\ldots^\smallfrown(c^1_i)_{i<\omega}\models
q_{\eta(\lambda-1)}^{\otimes\omega}\otimes\ldots\otimes q_{\eta(1)}^{\otimes\omega}|_{M^1}.$$
We have 
$$(b^{\lambda-1}_i)_{i<\omega}^\smallfrown\ldots^\smallfrown(b^1_i)_{i<\omega}\equiv_{M^0(b^0_i)_{i<\omega}}
(c^{\lambda-1}_i)_{i<\omega}^\smallfrown\ldots^\smallfrown(c^1_i)_{i<\omega}.$$
Now, let $M^2\preceq\FC$ be $|M^1|^+$-saturated such that $M^1(c^1_i)_{i<\omega}\subseteq M^2$
and let us choose one more Morley sequence
$$(d^{\lambda-1}_i)_{i<\omega}^\smallfrown\ldots^\smallfrown(d^2_i)_{i<\omega}\models
q_{\eta(\lambda-1)}^{\otimes\omega}\otimes\ldots\otimes q_{\eta(2)}^{\otimes\omega}|_{M^2}.$$
We see that
$$(d^{\lambda-1}_i)_{i<\omega}^\smallfrown\ldots^\smallfrown(d^2_i)_{i<\omega}\equiv_{M^1(c^1_i)_{i<\omega}}(c^{\lambda-1}_i)_{i<\omega}^\smallfrown\ldots^\smallfrown(c^2_i)_{i<\omega},$$
and so
\begin{IEEEeqnarray*}{rCl}
(d^{\lambda-1}_i)_{i<\omega}^\smallfrown\ldots^\smallfrown(d^2_i)_{i<\omega}^\smallfrown(c^1_i)_{i<\omega}^\smallfrown(b^0_i)_{i<\omega} &\equiv_{M^0}& (c^{\lambda-1}_i)_{i<\omega}^\smallfrown\ldots^\smallfrown(c^1_i)_{i<\omega}^\smallfrown(b^0_i)_{i<\omega} \\
&\equiv_{M^0}& (b^{\lambda-1}_i)_{i<\omega}^\smallfrown\ldots^\smallfrown(b^0_i)_{i<\omega}
\end{IEEEeqnarray*}
Set $(e^2_i)_{i<\omega}:=(d^2_i)_{i<\omega}$, $(e^1_i)_{i<\omega}:=(c^1_i)_{i<\omega}$ and $(e^0_i)_{i<\omega}:=(b^0_i)_{i<\omega}$.

Continuing this process we will obtain a sequence of models $M^0\preceq M^1\preceq\ldots\preceq M^{\lambda-1}$ and Morley sequences $(e^{\alpha}_i)_{i<\omega}\models q_{\eta(\alpha)}^{\otimes\omega}|_{M^{\alpha}}$, where $\alpha<\lambda$, such that
$M^{\alpha}(e^{\alpha}_i)_{i<\omega}\subseteq M^{\alpha+1}$ and
$$(e^{\lambda-1}_i)_{i<\omega}^\smallfrown\ldots^\smallfrown(e^0_i)_{i<\omega}\equiv_M
(b^{\lambda-1}_i)_{i<\omega}^\smallfrown\ldots^\smallfrown(b^0_i)_{i<\omega}.$$
Since $\dom(\pi(x))\subseteq M=M^0$, we have also that
$\pi(x)\,\cup\,\{\varphi_{\eta(\alpha)}(x;e^{\alpha}_0)\;|\;\alpha<\lambda\}$ is consistent
and that $\{\varphi_{\eta(\alpha)}(x;e^{\alpha}_i)\;|\;i<\omega\}$ is inconsitent for each $\alpha<\lambda$.
Therefore all the conditions of Lemma \ref{lemma:Qrank.NSOP1} are satisfied for $((e^{\alpha}_i)_{i<\omega},M^{\alpha})_{\alpha<\lambda}$.
\end{proof}

\begin{example}\label{ex:finite.not.1}
Let us provide an example of a situation when the rank of the home sort is finite, but strictly bigger than $1$. Let $k$ be any natural number greater than $1$ and let $p$ be equal to zero or to a prime number distinct from 2.  Consider the 2-sorted theory $T_m$ of $m$-dimensional vector spaces equipped with a non-degenerate symmetric bilinear form (see \cite[Chapter 10]{bilinear}). Let $x$ and $y$ be single vector variables, $\phi(x,y)=(x\perp y\wedge x\neq 0)$, and let $q(y)$ be the generic type in the vector sort $V$ (so $q$ is $\emptyset$-invariant). Put $Q=\{(\phi(x,y),q(y))\}$. We will show that $D_Q(\{x=x\})=k-1$. 

Let $(v_i)_{i<\omega}$ be a Morley sequence in $q(y)$. Then in particular $v_0\dots v_{k-1}$ are linearly independent, which easily implies that $\bigwedge_{i<k}\phi(x,v_i)$ is inconsistent. Thus $D_Q(\{x=x\})\leq k-1$.
For the other inequality, put $M_i=\acl(a_{ik},a_{ik+1},\dots,a_{ik+k-1})$ for $i<k$ and note that $b^i:=v_{(i+1)k}\models q|_{M_i}$ (as $b^i\ind M_i$ and $b^i\models q|_{\emptyset}$), and each $M_i$ is an elementary submodel by quantifier elimination. Also, the sequence $(v_{(i+1)k},v_{(i+1)k+1},\dots)$ is Morley over $M_i$, and it witnesses that $\phi(x,b^i)$ Kim-divides over $M_i$ for each $i<k$. Finally, $\bigwedge_{i<k-1} \phi(x,b^i)$ is consistent, as there is a non-zero vector orthogonal to $b^0,\dots,b^{i-2}$. This shows that $D_Q(\{x=x\})=k-1$. 
\end{example}

\begin{remark}
Note that if $T$ is an NSOP$_1$ theory (as we assume here) and for some $k<\omega$, a formula $\varphi(x,y)$,
$M\preceq\FC$, some $M$-invariant $q(y)\in S(\FC)$ and some $(b_i)_{i<\omega}\models q^{\otimes\omega}|_M$,
the set $\{\varphi(x,b_i)\;|\;i<\omega\}$ is $k$-inconsistent but not $(k-1)$-inconsistent, 
then $D_{((\varphi,q))}(x=x)=k-1$ (as in Example \ref{ex:finite.not.1}). To see this, consider a linear order $I$ being $(k-1)$-many copies of $\omega$ (one after another one) and $(c_i)_{i\in I}\models q^{\otimes I}|_M$ and use Lemma \ref{lemma:technical}. In other words, for $\pi(x):=\{x=x\}$ the situation is quite simple and either $D_{((\varphi,q))}(\pi)=0$ or
$D_{((\varphi,q))}(\pi)=C_{\varphi,q}$ provided $C_{\varphi,q}<\omega$.
\end{remark}

\begin{lemma}\label{lemma:implies}
If $\pi\vdash \pi'$ then 
$$D_Q(\pi)\leqslant D_{Q}(\pi').$$
\end{lemma}

\begin{proof}
Assume that $D_Q(\pi)\geqslant \lambda\in\mathbb{N}_{>0}$, then by Lemma \ref{lemma:technical}
there exist $\eta\in n^{\lambda}$, $(b^{\alpha}_i)_{\alpha<\lambda,i<\omega}$ and $M\preceq\FC$ such that
\begin{enumerate}
    \item $\dom(\pi)\subseteq M$,
    \item $q_1,\ldots,q_n$ are $M$-invariant,
    \item $\pi(x)\cup\{\varphi_{\eta(\alpha)}(x;b^{\alpha}_0)\;|\;\alpha<\lambda\}$ is consistent,
    \item $\{\varphi_{\eta(\alpha)}(x;b^{\alpha}_i)\;|\;i<\omega\}$ is inconsistent for each $\alpha<\lambda$,
    \item $$(b^{\lambda-1}_i)_{i<\omega}^\smallfrown\ldots^\smallfrown(b^0_i)_{i<\omega}\models q_{\eta(\lambda-1)}^{\otimes\omega}\otimes\ldots\otimes q_{\eta(0)}^{\otimes\omega}|_{M}.$$
\end{enumerate}
Let $N\preceq\FC$ contain $M$ and $\dom(\pi')$, and let 
$$(c^{\lambda-1}_i)_{i<\omega}^\smallfrown\ldots^\smallfrown(c^0_i)_{i<\omega}\models q_{\eta(\lambda-1)}^{\otimes\omega}\otimes\ldots\otimes q_{\eta(0)}^{\otimes\omega}|_{N}.$$
Then naturally $\dom(\pi')\subseteq N$ and $q_1,\ldots,q_n$ are $N$-invariant.
Because 
$$(b^{\lambda-1}_i)_{i<\omega}^\smallfrown\ldots^\smallfrown(b^0_i)_{i<\omega}\equiv_M
(c^{\lambda-1}_i)_{i<\omega}^\smallfrown\ldots^\smallfrown(c^0_i)_{i<\omega},$$
we have also that $\{\varphi_{\eta(\alpha)}(x;c^{\alpha}_i)\;|\;i<\omega\}$ is inconsistent for each $\alpha<\lambda$, and that $\pi(x)\cup\{\varphi_{\eta(\alpha)}(x;c^{\alpha}_0)\;|\;\alpha<\lambda\}$ is consistent.
Moreover, $\pi\vdash\pi'$ implies that
$\pi'(x)\cup\{\varphi_{\eta(\alpha)}(x;c^{\alpha}_0)\;|\;\alpha<\lambda\}$ is consistent.
Hence, Lemma \ref{lemma:technical} gives us $D_Q(\pi')\geqslant\lambda$.
\end{proof}

\begin{lemma}\label{lemma:Qrank.vee}
$D_Q(\pi\cup\{\bigvee\limits_{j\leqslant m}\psi_j\})=\max\limits_{j\leqslant m} D_Q(\pi\cup\{\psi_j\})$.
\end{lemma}

\begin{proof}
Because $\pi\cup\{\psi_i\}\vdash\pi\cup\{\bigvee_{j}\psi_j\}$ for each $i\leqslant m$, 
Lemma \ref{lemma:implies} gives us that
$$\max\limits_{j\leqslant m} D_Q(\pi\cup\{\psi_j\})\leqslant D_Q(\pi\cup\{\bigvee\limits_{j\leqslant m}\psi_j\}).$$
Hence it is enough to show that
$$D_Q(\pi\cup\{\bigvee\limits_{j\leqslant m}\psi_j\})\geqslant\lambda\quad\Rightarrow\quad
\max\limits_{j\leqslant m} D_Q(\pi\cup\{\psi_j\})\geqslant\lambda.$$
Let $D_Q(\pi\cup\{\bigvee_{j}\psi_j\})\geqslant\lambda$, i.e. 
there exists $\eta\in n^\lambda$ and $(b^{\alpha},M^{\alpha})_{\alpha<\lambda}$ as in Definition \ref{def:Qrank}, in particular
\begin{itemize}
    \item $\dom(\pi\cup\{\bigvee_{j}\psi_j\})\subseteq M^0$,
    \item $\pi\cup\{\bigvee_j\psi_j\}\cup\{\varphi_{\eta(\alpha)}(x;b^{\alpha})\;|\;\alpha<\lambda\}$ is consistent.
\end{itemize}
Thus there is $i_0\leqslant m$ such that 
$\pi\cup\{\psi_{i_0}\}\cup\{\varphi_{\eta(\alpha)}(x;b^{\alpha})\;|\;\alpha<\lambda\}$ is consistent.
Because $\dom(\pi\cup\{\psi_{i_0}\})\subseteq\dom(\pi\cup\{\bigvee_j\psi_j\})\subseteq M^0$, we have that $D_Q(\pi\cup\{\psi_{i_0}\})\geqslant\lambda$ and so also 
$\max_{j} D_Q(\pi\cup\{\psi_j\})\geqslant\lambda$.
\end{proof}

Let us recall that we are working in a theory $T$ which is NSOP$_1$.

\begin{lemma}\label{lemma:finite.character}
Assume that $q_0(y_0)=\ldots=q_{n-1}(y_{n-1})=q(y)$ (in $Q$).
\begin{enumerate}
    \item Let $\{\pi_{\beta}\}$ lists all finite subsets of $\pi$. If for each $\beta$ we have that $D_Q(\pi_{\beta})\geqslant \lambda<\omega$, then $D_Q(\pi)\geqslant \lambda$.

    \item We can always find a finite $\pi_0\subseteq\pi$ such that $D_Q(\pi_0)=D_Q(\pi)$.
\end{enumerate}
\end{lemma}

\begin{proof}
Because for each $\beta$ we have $D_Q(\pi_{\beta})\geqslant \lambda$,
by Lemma \ref{lemma:technical}, for each $\beta$
there exists a function $\eta_{\beta}\in n^{\lambda}$,
a sequence of sequences $(b^{\beta,\alpha}_i)_{\alpha<\lambda,i<\omega}$ and
a model $M^{\beta}\preceq\FC$ such that
\begin{enumerate}
    \item $\dom(\pi_{\beta})\subseteq M^{\beta}$,
    \item $q$ is $M^{\beta}$-invariant,
    \item $\pi_{\beta}(x)\cup\{\varphi_{\eta_{\beta}(\alpha)}(x;b^{\beta,\alpha}_0)\;|\;\alpha<\lambda\}$ is consistent,
    \item $\{\varphi_{\eta_{\beta}(\alpha)}(x;b^{\beta,\alpha}_i)\;|\;i<\omega\}$ is inconsistent for each $\alpha<\lambda$,
    \item $$(b^{\beta,\lambda-1}_i)_{i<\omega}^\smallfrown\ldots^\smallfrown(b^{\beta,0}_i)_{i<\omega}\models q^{\otimes\omega}\otimes\ldots\otimes q^{\otimes\omega}|_{M^{\beta}}.$$
\end{enumerate}
Let $N\preceq\FC$ be such that $\bigcup\limits_{\beta}M^{\beta}\subseteq N$.
Moreover, let us pick up the following Morley sequences
$$(c^{\lambda-1}_i)_{i<\omega}^\smallfrown\ldots^\smallfrown(c^0_i)_{i<\omega}\models q^{\otimes\omega}\otimes\ldots\otimes q^{\otimes\omega}|_N,$$
and, without loss of generality, let $\{\varphi_{\eta_{\beta}(\alpha)}(x;y)\;|\;\beta,\alpha<\lambda\}=\{\varphi_0(x;y),\ldots,\varphi_{r-1}(x;y)\}$
for some $s\leqslant n$. We introduce $\psi(x;y):=\bigvee\limits_{i<r}\varphi_i(x;y)$, and note that
$$\pi(x)\,\cup\,\{\psi(x;c^{\alpha}_0)\;|\;\alpha<\lambda\}$$
is consistent (otherwise, by compactness, $\pi_{\beta}(x)\,\cup\,\{\psi(x;c^{\alpha}_0)\;|\;\alpha<\lambda\}$
is inconsistent for some $\beta$, which is impossible, since
$$(c^{\lambda-1}_i)_{i<\omega}^\smallfrown\ldots^\smallfrown(c^0_i)_{i<\omega}\equiv_{M^{\beta}}(b^{\beta,\lambda-1}_i)_{i<\omega}^\smallfrown\ldots^\smallfrown(b^{\beta,0}_i)_{i<\omega}$$
and $\pi_{\beta}(x)\cup\{\varphi_{\eta_{\beta}(\alpha)}(x;b^{\beta,\alpha}_0)\;|\;\alpha<\lambda\}$ is consistent).
Consider $d\models \pi(x)\,\cup\,\{\psi(x;c^{\alpha}_0)\;|\;\alpha<\lambda\}$, then for each $\alpha<\lambda$
there is $i_{\alpha}<r$ such that $\models\varphi_{i_{\alpha}}(d,c^{\alpha}_0)$. Let $\eta\in n^{\lambda}$ be given by $\eta:\alpha\mapsto i_{\alpha}$.

As we want to use Lemma \ref{lemma:technical} to show that $D_Q(\pi)\geqslant\lambda$, and we have already defined $\eta$, $N\preceq\FC$ and $(c^{\alpha}_i)_{\alpha<\lambda,i<\omega}$, we need to verify whether all the five conditions from Lemma \ref{lemma:technical} hold.
Obviously, $\dom(\pi)\subseteq N$ and $q$ is $N$-invariant, so we have the first and the second condition. The fifth condition is naturally satisfied by the choice of $(c^{\alpha}_i)_{\alpha<\lambda,i<\omega}$. The third condition says that 
$\pi(x)\,\cup\,\{\varphi_{\eta(\alpha}(x,c^{\alpha}_0)\;|\;\alpha<\lambda\}$ is consistent, which is witnessed by element $d$. For the fourth condition, we need to note that $\{\varphi_{\eta(\alpha)}(x,c^{\alpha}_i)\;|\;i<\omega\}$ is inconsistent for every $\alpha<\lambda$.
Because $\eta(\alpha)< r$, there exist $\beta$ and $\alpha'<\lambda$ such that $\eta(\alpha)=\eta_{\beta}(\alpha')$. We know that $\{\varphi_{\eta_{\beta}(\alpha')}(x,b^{\beta,\alpha'}_i)\;|\;i<\omega\}$ is inconsistent,
that $(b^{\beta,\alpha'}_i)_{i<\omega}\models q^{\otimes\omega}|_{M^{\beta}}$ and
that $(c^{\alpha}_i)_{i<\omega}\models q^{\otimes\omega}|_{M^{\beta}}$. 
Thus also $\{\varphi_{\eta(\alpha)}(x,c^{\alpha}_i)\;|\;i<\omega\}$ is inconsistent.
\end{proof}

\begin{corollary}\label{cor:extensions}
Let 
$q_0(y_0)=\ldots=q_{n-1}(y_{n-1})=q(y)$ (in $Q$)
and let $\pi(x)$ be a partial type over $A$.
Then there exists $p(x)\in S(A)$ extending $\pi(x)$ such that $D_Q(p)=D_Q(\pi)$.
\end{corollary}

\begin{proof}
The proof is completely standard, but let us sketch it anyway.
Consider
$$\tilde{\pi}(x):=\pi(x)\,\cup\,\{\neg\psi(x)\in L(A)\;|\;D_Q(\pi(x)\,\cup\,\{\psi(x)\})<D_Q(\pi(x))\}.$$
By Lemma \ref{lemma:implies} and Lemma \ref{lemma:Qrank.vee}, the set $\tilde{\pi}$ is a partial type over $A$.
Let $p(x)\in S(A)$ be any extension of $\tilde{\pi}$.

If $D_Q(p)<D_Q(\pi)$ then by Lemma \ref{lemma:finite.character} there exists a finite subset $p_0(x)\subseteq p(x)$ such that $D_Q(p_0)<D_Q(\pi)$. Let $\psi(x):=\bigwedge p_0(x)\in L(A)$, we have that
$$D_Q(\pi(x)\cup\{\psi(x)\})\leqslant D_Q(\psi(x))\leqslant D_Q(p_0(x))<D_Q(\pi(x)).$$
Thus $\neg\psi(x)\in \tilde{\pi}(x)\subseteq p(x)$ and we got a contradiction with $\psi(x)\in p(x)$.
\end{proof}

Let $\Delta=\{\varphi_1(x;y),\ldots,\varphi_n(x;y)\}$ and $1<k<\omega$.
Recall that there is a local rank used in simple theories, denoted $D(\,\cdot\,,\Delta,k)$ (cf. Chapter 3 in \cite{casasimpl}). This local rank may be used to characterize simplicity as:
$T$ is simple if and only if $D(\{x=x\},\{\varphi\},k)<\omega$ for all $\varphi$ and $k$ (e.g. Proposition 3.13 in \cite{casasimpl}). 
As our rank is also local, we can treat our rank as
an analogon of the local rank $D(\,\cdot\,,\Delta,k)$. Let us compare now the both local ranks.

\begin{remark}\label{rem:compare}
For each $Q=((\varphi_0,q_0),\ldots,(\varphi_{n-1},q_{n-1}))$
and any partial type $\pi$,
there exists $K<\omega$ such that
for any $k\geqslant K$ we have
$$D_Q(\pi)\leqslant D(\pi,\{\varphi_0,\ldots,\varphi_{n-1}\},k).$$
\end{remark}

\begin{proof}
As $T$ is NSOP$_1$, there is some $\lambda<\omega$ such that $D_Q(\pi)=\lambda$.
Let $\eta$ and $(b^\alpha,M^\alpha)_{\alpha<\lambda}$ be as in Definition \ref{def:Qrank}.
For each $\alpha<\lambda$ there exists $k_{\alpha}<\omega$ such that $\varphi_{\eta(\alpha)}(x;b^{\alpha})$
$k_{\alpha}$-divides over $M^{\alpha}$. 
Set $K:=\max\{k_0,\ldots,k_{\lambda-1}\}$. Then, by definition, $D(\pi,\{\varphi_0,\ldots,\varphi_{n-1}\},k)\geqslant\lambda$, provided $k\geqslant K$.
\end{proof}

On the other hand, if $T$ is NSOP$_1$ but not simple, then there must be a formula $\varphi$ and some $k_0<\omega$
such that $D(\{x=x\},\{\varphi\},k_0)\geqslant\omega$.
Thus for every $K<\omega$ there exists $k\geqslant K$ (e.g. $k=K+k_0$) such that $D(\{x=x\},\{\varphi\},k)\geqslant\omega$ but,
$D_{((\varphi,q))}(\{x=x\})<\omega$ for any choice of $q\in S(\FC)$. Therefore sharp inequality in Remark \ref{rem:compare} happens outside of the class of simple theories. One could ask about equality under the assumption on simplicity.
The following counterexample, which is even stable, leaves no doubt.

\begin{example}
Let $T$ be the theory of an equivalence relation $E$ with infinitely many classes all of which are infinite. It is well-known that $T$ is $\omega$-stable of Morley rank $2$. Let $\varphi_0(x,y)=E(x,y)$ and $\varphi_1(x,y)=(x=y)$. Then it is easy to see that for any $k>1$ we have that $D(\{x=x\},\{\varphi_0(x,y),\varphi_1(x,y)\},k)=2$.
Now, fix two arbitrary invariant global types $q_0(x),q_1(x)$ and put $Q=\{(\varphi_0,q_0),(\varphi_1,q_1)\}$. We claim that $D_Q(\{x=x\})=1$. Obviously $D_Q(\{x=x\})\geqslant 1$, so suppose for a contradiction that $D_Q(\{x=x\})\geqslant 2$ witnessed by $M^0,M^1,b^0,b^1$ and $\eta:2 \to 2$. The case where $\eta(0)=1$ can be excluded immediately, so assume $\eta(0)=0$.
Observe that $\models E(b^0,b^1)$, as otherwise $\varphi_{\eta(0)}(x,b^0)\wedge \varphi_{\eta(1)}(x,b^1)$ would be inconsistent (note $\varphi_i(x,y)\vdash E(x,y)$ for $i=0,1$).
On the other hand, as $b^1\models q_{\eta(1)}|_{M^0b^0}$, we have in particular that $b^1\ind_{M^0} b^0$, so $E(b^0,M^0)\neq \emptyset$, which contradicts that $\varphi_0(x,b^0)$ divides over $M^0$.
\end{example}

\subsection{Rank vs Kim-independence}\label{sec:rank.ind}
We know that Kim-generics do not exist in the theory of infinite dimensional vector spaces with a bilinear form (see Proposition 8.15 in \cite{bilinear}), which is NSOP$_1$. 
Since in the case of simple theories, a notion of finite local rank, which is compatible with forking and somehow invariant under shifts by elements of a definable group (e.g. Fact 3.7 in \cite{Anand98}), would lead to the existence of forking generics (see Lemma 3.8 in \cite{Anand98}), we probably should not expect that a notion of finite local rank in the case of NSOP$_1$ theories will be compatible with Kim-forking and invariant under shifts by group elements (otherwise one could try to prove existence of Kim-generics, which does always hold). 
Anyway, it seems that our notion of rank does not behave well under shifts by elements of some definable group, so does not immediately exclude compatibility of the rank with Kim-forking. Here, we study this problem and relate our results to an important question from \cite{CKR}.

\begin{lemma}\label{lemma:DQ.to.K-ind}
Let $M\preceq N\preceq\FC$ and $a\in\FC$.
If $D_Q(\tp(a/M))=D_Q(\tp(a/N))$ for each $M$-invariant $Q$ such that $|Q|=1$,
then $a\ind^K_M N$.
\end{lemma}

\begin{proof}
Assume that $a\nind^K_M N$, which means that $\tp(a/N)$ Kim-divides over $M$.
Let $\varphi(x,b)\in \tp(a/N)$ Kim-divide over $M$. There exists an $M$-invariant $q(y)\in S(\FC)$ extending $\tp(b/M)$ such that $\varphi(x,b)$ $q$-divides over $M$. We set $Q=\big((\varphi(x;y),q(y))\big)$.
By Lemma \ref{lemma:implies} and Corollary \ref{cor:DQ.finite}, it follows that
$$D_Q(\tp(a/N))\leqslant D_Q(\tp(a/M)\;\cup\;\{\varphi(x;b)\})=:\lambda<\omega.$$
Therefore there exists a sequence $(b^{\alpha},M^{\alpha})_{\alpha<\lambda}$ such that
\begin{enumerate}
\item $Mb\subseteq M^0$,
\item $q$ is $M^0$-invariant,
\item $M^{\alpha}\preceq\FC$ for each $\alpha<\lambda$, $(M^{\alpha})_{\alpha<\lambda}$ is continuous, and each $M^{\alpha+1}$ is $|M^{\alpha}|^+$-saturated,
\item $b^{\alpha}M^{\alpha}\subseteq M^{\alpha+1}$ for each $\alpha+1<\lambda$,
\item $b^{\alpha}\models q|_{M^{\alpha}}$ for each $\alpha<\lambda$,
\item $\tp(a/M)\;\cup\;\{\varphi(x;b)\}\;\cup\;\{\varphi(x;b^{\alpha})\;|\;\alpha<\lambda\}$ is consistent,
\item $\varphi(x;b^{\alpha})$ Kim-divides over $M^\alpha$ for each $\alpha<\lambda$.
\end{enumerate}
By Lemma \ref{lemma:technical}, we can modify $M^0$ and so assume that $M^0$ is $|M|^+$-saturated, which we do. We set $M^{-1}:=M$ and $b^{-1}:=b$. Checking that $(b^{\alpha},M^{\alpha})_{-1\leqslant \alpha<\lambda}$ witnesses that $D_Q(\tp(a/M))\geqslant\lambda+1$ is routine.
\end{proof}

\begin{lemma}\label{lemma:ind.to.DQ}
Let $T$ be NSOP$_1$ with existence.
Assume that $a\in\FC$, $M\preceq N\preceq\FC$,
 $N$ is $|M|^+$-saturated,
  $q(y)\in S(\FC)$ is $M$-invariant, and that $q^{\otimes I}|_M$ is stationary for any order $I$ of the form $\omega^\smallfrown\ldots^\smallfrown\omega$ (finitely many copies of $\omega$).
Let $Q=\big((\varphi(x;y,),q(y))\big)$.
If $a\ind_M N$ then $D_Q(a/M)=D_Q(a/N)$.
\end{lemma}

\begin{proof}
Let $D_Q(a/M)=:\lambda<\omega$, it means by Lemma \ref{lemma:Qrank.NSOP1} that there 
exists $\big( (b^{\alpha}_i)_{i<\omega}, M^{\alpha}\big)_{\alpha<\lambda}$ such that
\begin{enumerate}
\item $M\subseteq M^0$,
\item $q$ is $M^0$-invariant (which comes for free as $q$ is $M$-invariant and $M\subseteq M^0$),
\item $M^{\alpha}\preceq\FC$, $(M^{\alpha})_{\alpha<\lambda}$ is continuous, $M^{\alpha+1}$ is $|M^{\alpha}|^+$-saturated,
\item $b^{\alpha}_{<\omega}M^{\alpha}\subseteq M^{\alpha+1}$,
\item $b^{\alpha}_{<\omega}\models q^{\otimes\omega}|_{M^{\alpha}}$,
\item $\tp(a/M)\;\cup\;\{\varphi(x;b^{\alpha}_0)\;|\;\alpha<\lambda\}$ is consistent,
\item $\{\varphi(x;b^{\alpha}_i)\;|\;i<\omega\}$ is inconsistent.
\end{enumerate}
\
\\
\textbf{Step 1.} 
\\
Let $d\models \tp(a/M)\;\cup\;\{\varphi(x;b^{\alpha}_0)\;|\;\alpha<\lambda\}$ and let $N''\preceq\FC$ be such that
$aN\equiv_M dN''$. 
By the existence axiom for forking independence, there exists $N'\equiv_{Md}N''$ such that
$$M^{<\lambda}b^{<\lambda}_{<\omega}\ind_{Md} N'.$$
Because $a\ind_M N$, we have that $d\ind_M N''$ and then also that $d\ind_M N'$.
As $N'$ is $|M|^+$-saturated, Proposition 5.4 from \cite{casasimpl} assures us that
$d\ind_M N'$ and $M^{<\lambda}b^{<\lambda}_{<\omega}\ind_{Md} N'$ combine into
$dM^{<\lambda}b^{<\lambda}_{<\omega}\ind_{M} N'$. Then monotonicity of $\ind$ gives us $b^{<\lambda}_{<\omega}\ind_{M} N'$.
\
\\
\textbf{Step 2.}
\\
Because $aN\equiv_M dN''\equiv_M dN'$ and $q(y)$ is $M$-invariant, we have that $D_Q(a/N)=D_Q(d/N')$, so it is enough to show that $D_Q(d/N')\geqslant\lambda$. Note that $\tp(d/N')\;\cup\;\{\varphi(x;b^{\alpha}_0)\;|\;\alpha<\lambda\}$ is consistent.
\
\\
\textbf{Step 3.}
\\
Note that 
$$b^{<\lambda}_{<\omega}=(b^{\lambda-1}_i)_{i<\omega}^\smallfrown\ldots^\smallfrown(b^0_i)_{i<\omega}\models
q^{\otimes\omega}\otimes\ldots\otimes q^{\otimes\omega}|_M.$$
Since $b^{<\lambda}_{<\omega}\ind_M N'$ and 
the type $q^{\otimes\omega}\otimes\ldots\otimes q^{\otimes\omega}|_M$ ($\lambda$-many repetitions) is
stationary, we have that also 
$$b^{<\lambda}_{<\omega}=(b^{\lambda-1}_i)_{i<\omega}^\smallfrown\ldots^\smallfrown(b^0_i)_{i<\omega}\models
q^{\otimes\omega}\otimes\ldots\otimes q^{\otimes\omega}|_{N'}.$$
\
\\
\textbf{Step 4.}
\\
Now, Lemma \ref{lemma:technical} gives us easily
$D_Q(d/N')\geqslant\lambda$.
\end{proof}

\begin{definition}
Consider $p(x)\in S(A)$. We call $p(x)$ \emph{Kim-stationary} if for each $B\supseteq A$ there is a unique Kim-nonforking extension of $p$ over $B$.
\end{definition}

\noindent
Repeating the proof of \cite[Lemma 11.6]{casasimpl}  we get:
\begin{fact}\label{stat_monot}
(arbitrary $T$) Let $\ind^*$ be any invariant ternary relation satisfying symmetry [over models] and extension [over models].
Then for any parameter set [any model] $M$ and $a,b$, if $\tp(ab/M)$ is 
$\ind^*$-stationary then $tp(a/M)$ is $\ind^*$-stationary.
\end{fact}
\begin{proof}
Let $N\succeq M$ and let $a'\equiv_M a''\equiv_M a$ be such that $a'\ind^*_M N$ and $a''\ind^*_M N$.
By extension, symmetry, and invariance,  there are $b',b''$ such that $a'b'\equiv_M a''b''\equiv_M ab$, $a'b'\ind^*_M N$ and $a''b''\ind^*_M N$. Then by $\ind^*$-stationarity of $\tp(ab/M)$ we have that $a'b'\equiv_N a''b''$, so in particular $a'\equiv_N a''$. 
\end{proof}

\begin{remark}\label{finchar}
Let $p(x)$ be a complete type over $M\subseteq\FC$ with $x$ possibly infinite. 
\begin{enumerate}
    \item Let $T$ be simple, then $p(x)$ is stationary if and only if $p|_{x_0}$ is stationary for every finite subtuple $x_0$ of $x$.
    \item ($T$ being NSOP$_1$) 
    Let $M\preceq\FC$. We have that 
    $p(x)$ Kim-stationary if and only if $p|_{x_0}$ is Kim-stationary for every finite subtuple $x_0$ of $x$.
\end{enumerate}

\end{remark}
\begin{proof}
We start with the proof of (1).
The implication from left to right follows from Fact \ref{stat_monot}. Assume $p|_{x_0}$ is stationary for every finite subtuple $x_0$ of $x$ and let $q_0,q_1$ be global non-forking extensions of $p$. Then for every finite subtuple $x_0$ of $x$ we have that $q_0|_{x_0}$ and $q_1|_{x_0}$ are non-forking extensions of $p|_{x_0}$, so by stationarity of $p|_{x_0}$ they are equal, hence $q_0=q_1$ and $p$ is stationary. 

The argument for Kim-stationarity is exactly the same.
\end{proof}

\begin{lemma}\label{lemma:stationary}
 Let $A\subseteq\FC$, let $(I,<)$ be an infinite linear ordering without a maximal element, and let $q(y)\in S(\FC)$ be an $A$-invariant type.
\begin{enumerate}
    \item Let $T$ be simple. If $q|_A$ is stationary then also $q^{\otimes\omega}|_A$ is stationary.
    
    \item (any $T$) 
    If $q^{\otimes\omega}|_A$ is stationary then $q^{\otimes I}|_A$ is stationary.
    
    \item (any $T$)
    If $q^{\otimes\omega}|_A$ is Kim-stationary then $q^{\otimes I}|_A$ is Kim-stationary.
\end{enumerate}
\end{lemma}

\begin{proof}
We start with the proof of (1).
By Remark \ref{finchar}, it is enough to show that for each $n<\omega$, the type $q^{\otimes n}|_A$ is stationary. This can be shown inductively by the use of the following claim:
\ 
\\
\textbf{Claim}.
If $p(x)\in S(\FC)$ is $A$-invariant, such that $p|_A$ is stationary, then $q\otimes p|_A$ is stationary.
\ 
\\
\textit{Proof of the claim:}
Let $B\supseteq A$, let $ab\models q\otimes p|_A$ be such that $ab\ind_A B$, and let $a'b'\models q\otimes p|_B$. Our goal is $ab\models q\otimes p|_B$.
As $ab\ind_A B$, we have $b\ind_A B$. Since $b'\models p|_B$ and $p$ is $A$-invariant, also $b'\ind_A B$.
Stationarity of $p|_A$ implies that there exists $f\in\aut(\FC/B)$ such that $f(b)=b'$,
thus $b\models p|_B$.
Because $a\models q|_{Ab}$, we have that $f(a)\models q|_{Ab'}$.
Since $q$ is $A$-invariant, we obtain that $f(a)\ind_A b'$.
On the other hand, $ab\ind_A B$ changes via $f$ into $f(a)b'\ind_A B$, then symmetry, base monotonicity, normality and monotonicity of $\ind$ give us $f(a)\ind_{Ab'}Bb'$. By transitivity of $\ind$,
$f(a)\ind_A b'$ and $f(a)\ind_{Ab'}Bb'$ imply that $f(a)\ind_A Bb'$, which
by the fact that $a'\models q|_{Bb'}$ (so $a'\ind_A Bb'$) and
by stationarity of $q|A$
implies $f(a)\equiv_{Bb'}a'$.
As $f(a)\models q|_{Bb'}$, we have that $a\models q|_{Bb}$, thus $ab\models q\otimes p|_B$ and we end the proof of the claim.

Now, we are moving to the proof of (2).
Assume that $\bar{a}=(a_i)_{i\in I}\models q^{\otimes I}|_A$ and $\bar{b}=(b_i)_{i\in I}\models q^{\otimes I}|_A$ are such that $\bar{a}\ind_A B$ and $\bar{b}\ind_A B$ for some $B\supseteq A$.
We need to show that $\tp(\bar{a}/B)=\tp(\bar{b}/B)$, which holds if and only if for each $n<\omega$ and each $i_1,\ldots,i_n\in I$, such that $i_1<\ldots<i_n$, we have $\tp(a_{i_n}\ldots a_{i_1}/B)=\tp(b_{i_n}\ldots b_{i_1}/B)$.
Consider such $i_1,\ldots,i_n\in I$ and choose any infinite sequence $I_0\subseteq I$ starting with $(i_1,\ldots,i_n)$ which has the order type of $(\omega,<)$.
Let $a_{I_0}:=(a_i)_{i\in I_0}$ and $b_{I_0}:=(b_i)_{i\in I_0}$.
Then $a_{I_0}\models q^{\otimes\omega}|_A$, $b_{I_0}\models q^{\otimes\omega}|_A$.
Monotonicity of $\ind$ gives us that $a_{I_0}\ind_A B$ and $b_{I_0}\ind_A B$.
As $q^{\otimes\omega}|_A$ is stationary we have that $a_{I_0}\equiv_B b_{I_0}$, so in particular
$\tp(a_{i_n}\ldots a_{i_1}/B)=\tp(b_{i_n}\ldots b_{i_1}/B)$.

The proof of (3) is similar to the proof of (2), as the only property of $\ind$ used was monotonicity, which also holds for $\ind^K$.
\end{proof}

\begin{definition}
Let $A\subseteq \FC$ and let $q(y)\in S(\FC)$ be $A$-invariant.
\begin{enumerate}
    \item We call $q|_A$ \emph{strongly stationary} if the type
    $q^{\otimes\omega}|_A$ is stationary.
    
    \item We call $q|_A$ \emph{strongly Kim-stationary} if the type
    $q^{\otimes\omega}|_A$ is Kim-stationary.  
\end{enumerate}
\end{definition}

\begin{remark}
Let $A\subseteq \FC$ and let $q(y)\in S(\FC)$ be $A$-invariant,
and let $(I,<)$ be an infinite linear ordering without a maximal element.
\begin{enumerate}
    \item If the type $q|_A$ is strongly stationary then $q^{\otimes I}|_A$ is stationary.
    
    \item 
    If in addition $T$ is simple, then the type $q|_A$ is strongly stationary if and only if
    $q|_A$ is stationary.

    \item If the type $q|_A$ is strongly Kim-stationary then $q^{\otimes I}|_A$ is Kim-stationary.
    
    \item 
    If in addition $A\subseteq M\preceq\FC$, 
    the type $q^{\otimes\omega}|_M$ is Kim-stationary if and only if for each finite subtuple $\bar{x}$ of $(x_i)_{i<\omega}$,
    the type $\big(q^{\otimes\omega}|_M\big)_{\bar{x}}$ is Kim-stationary.
\end{enumerate}
\end{remark}

\begin{question}
Let $M\preceq\FC$ and let $q(y)\in S(\FC)$ be $M$-invariant such that $q|_M$ is Kim-stationary.
Is the type $q^{\otimes\omega}|_M$ Kim-stationary?
\end{question}

\begin{corollary}
Let $T$ be NSOP$_1$ with existence.
Assume that $a\in\FC$, $M\preceq N\preceq\FC$,
that $q(y)\in S(\FC)$ is $M$-invariant and let $Q=\big((\varphi(x;y,),q(y))\big)$.
\begin{enumerate}
    \item Assume that $N$ is $|M|^+$-saturated and $q|_M$ is strongly stationary, then
    $$a\ind_M N\qquad\Rightarrow\qquad D_Q(a/M)=D_Q(a/N).$$
    
    \item Assume that $T$ is simple and $q|_M$ is stationary, then
    $$a\ind_M N\qquad\Rightarrow\qquad D_Q(a/M)=D_Q(a/N).$$
    
    \item If $q|_M$ is strongly Kim-stationary, then
    $$a\ind^K_M N\qquad\Rightarrow\qquad D_Q(a/M)=D_Q(a/N).$$
\end{enumerate}
\end{corollary}

\begin{proof}
(1) is exactly the content of Lemma \ref{lemma:ind.to.DQ}.
For (2), we use Lemma \ref{lemma:stationary} and repeat the proof of Lemma \ref{lemma:ind.to.DQ}, where in Step 1, instead of the saturation assumption we use transitivity, which naturally holds in simple theories.

(3) also uses a modification of the proof of Lemma \ref{lemma:ind.to.DQ}.
More precisely, in Step 1, we use that $\ind$ implies $\ind^K$ and that transitivity (over arbitrary sets) holds in NSOP$_1$ with existence. We obtain $dM^{<\lambda}b^{<\lambda}_{<\omega}\ind^K_M N'$ and then $b^{<\lambda}_{<\omega}\ind^K_M N'$.
Then, Step 3 follows by the definition of strong Kim-stationarity. The rest (Step 2 and Step 4) remains the same.
\end{proof}

\noindent
It is worth to compare (2) and (3) in the above corollary with Theorem 4.7(2) from \cite{CKR}. 
We are aware that we introduced a different notion of rank, but because of that, we were able to drop the assumption on simplicity from \cite{CKR}[Theorem 4.7.(2)] at the cost of assuming strong stationarity of the type $q$ and so to provide partial answer to the counterpart to Question 4.9 from \cite{CKR} for our notion of rank.

In Example \ref{rank_drop}, we will observe that the stationarity assumption in Lemma \ref{lemma:ind.to.DQ} cannot be removed. Moreover, in Proposition \ref{prop:stationarity.bilinear}, we will see that in the case of vector spaces with a bilinear form, being strongly Kim-stationary is equivalent to being Kim-stationary.

\section{Forking in $T_{\infty}$}\label{sec:bilinear}
In this section we describe forking in the theory 
$T_\infty$ of vector spaces with a generic bilinear form, answering in particular a question  about equivalence of dividing and dividing finitely stated in the first paragraph of Section 12.5 in \cite{Granger}. Before that let us recall the basic definitions and provide some facts related to the previous sections.

Let $T_{\infty}$ be the theory of two-sorted vector spaces over an algebraically closed field with a sort $V$ for vectors and a sort $K$ for scalars, equipped with a non-degenerate symmetric (or alternating) bilinear form, as studied in \cite{Granger}. Then $T$ is NSOP$_1$ by \cite[Corollary 6.1]{ArtemNick}
and it has existence by \cite[Proposition 8.1]{bilinear}.
Fix a monster model $\FC=(V(\FC), K(\FC))$ of $T_{\infty}$.
If  $A\subseteq \FC\models T_\infty$, put
$\langle A \rangle:=\Lin_{K(\FC)}(V(A))$ and  let $A_K:=A\cap K(\FC)$.  
By \cite[Corollary 8.13]{bilinear}, for any sets $A,B,C$ we have that
$A\ind^K_C B$ if and only if 
\begin{itemize}
    \item $\langle AC \rangle\cap \langle BC \rangle =\langle C \rangle$ and
    \item $\dcl(AC)_K\ind ^{ACF}_{\dcl(C)_K}\dcl(BC)_K$,
\end{itemize}
where $\dcl(\ast)_K$ means $(\dcl(\ast))_K$ (for an algebraic description of $\dcl(*)$ see \cite[Proposition 9.5.1]{Granger}).
For a discussion about forking and dividing in $T_\infty$ see \cite[Subsection 12.3]{Granger}
and the rest of this section below.

\begin{example}\label{rank_drop}
Let $M\prec N\models T_\infty$, $v\in V(N)\setminus V(M)$ and let $q_0$ be the unique global $\ind^\Gamma$-generic type in $V$. Then  $q_0|_M(x)\cup \{x\perp v\}$ does not Kim-fork over $M$ so it has a realisation $w$ with $w\ind^K_M N$. Let $\varphi(x,y)$ express that $x\in \langle y\rangle \setminus \{0\}$ with $x$ and $y$ being single vector variables, and $Q=\{(\varphi(x,y),q_0(y))\}$. We claim that $D_Q(w/N)=0$ and $D_Q(w/M)\geq 1$, which shows that the \textbf{stationarity assumption in Lemma \ref{lemma:ind.to.DQ} cannot be removed}. Indeed, note that if $N\prec N_0$ and $b\models q_0|_{N_0}$ then $b$ is not orthogonal to $v$, so $\tp(w/N)\cup \varphi(x,b)$ is inconsistent, hence $D_Q(w/N)=0$. On the other hand, $\tp(w/M) \cup \varphi(x,b)$ is consistent and $\varphi(x,b)$ $q_0$-forks over $N_0$ (as the sets $\langle b_i\rangle\setminus \{0\}$ are pairwise disjoint for a Morley sequence $(b_i)_{i<\omega}$ in $q_0$ over $N_0$, as such a sequence needs to be linearly independent since $q_0$ is generic). Thus $D_Q(w/M)\geq 1$ (and it is actually easy to see that $D_Q(w/M)= 1$). 
\end{example}

\begin{proposition}\label{prop:stationarity.bilinear}
Let $p(x)=\tp(a/M)$ be a complete type over $M\models T_\infty$. Then $p(x)$ is Kim-stationary iff $a\subseteq K(\FC)\cup \langle V(M)\rangle$.
Hence $p(x)$ is Kim-stationary if and only if it is strongly Kim-stationary. 
\end{proposition}
\begin{proof}
By Remark \ref{finchar}, we may assume $x$ is a finite tuple of variables. Assume the right-hand side first. Then by compactness $p(x)\vdash x\subseteq V_0\cup K(\FC)$ for some finite-dimensional subspace $V_0\leq V(\FC)$. As $V_0\cup K(\FC)$ is stably embedded in $\FC$ and definable over $M$ in the pure field $K$, we can view $p(x)$ as a complete type over $V_0(M)\cup K(M)$ in the finite Morley rank structure $V_0\cup  K(\FC)$. Hence Kim-stationarity of $p(x)$ follows as in $V_0\cup K(\FC)$ Kim-independence coincides with forking independence and $V_0(M)\cup K(M)\prec V_0\cup K(\FC)$.

Now assume that $p(x)$ is stationary and let $p'(x_0)$ be its restriction to a single vector variable $x_0$. By Fact \ref{stat_monot}, $p'(x_0)$ is Kim-stationary. Let $a_0\models p'(x_0)$ be the coordinate of $a$ corresponding to $x_0$.
Suppose for a contradiction that $a_0\notin \langle M \rangle$. 
Let $v\in \V(\FC)$ be orthogonal to $M$ with $[v,v]=1$ and let $N\succ M$ be such that $v\in N$.
Let $r(x_0)$ be the unique $\Gamma$-independent extension of $p'(x_0)$ over $N$ and let $b\models r(x_0)$. Put $b':=b-[b,v]v$. 
As $b\ind^\Gamma _M N$, we get that $[b,v]\notin N$ (see (3) in  \cite[Definition 9.35]{kaplanramsey2017}), so in particular $[b,v]\neq 0$. Hence, as $b$ and $v$ are linearly independent over $K(M)$ (as $b\notin K(N)$), we have $b'\notin \langle N \rangle$. Also, for any $m\in V(M)$ we have $[b',m]=[b,m]-[b,v][v,m]=0$, and $[b',b']=[b,b]+[b,v][v,v]-2[b,v]^2\notin M$ as $[b,v]$ is transcendental over $K(N)$. Thus $b'\equiv_M b $ by quantifier elimination and $K(\dcl(b'N))\subseteq K(\dcl(b'N))$, so $b\ind^K_M N$ implies that $b'\ind^K_M N$ (see \cite[Corollary 8.13]{bilinear}).
But $[b',v]=[b,v]-[b,v][v,v]=0\neq [b,v]$, so $\tp(b/N)$ and $\tp(b'/N)$ are two distinct Kim-independent extensions of $p'(x_0)$, which contradicts Kim-stationarity of $p'(x_0)$.
\end{proof}

From now on, we assume the bilinear form is symmetric; the arguments in the alternating case are analogous. 
For $m<\omega$, $T_m$ denotes the theory of $m$-dimensional vector spaces over an algebraically closed field equipped with a non-degenerate symmetric bilinear form.
For the definitions of dividing finitely and $\Gamma$MS-dividing see Definitions 12.3.8 and 12.3.11 in \cite{Granger}. Note that, by Kim's Lemma, $\Gamma$MS-dividing is the same as Kim dividing.

By $\omega^*=\{i^*:i<\omega\}$ we will denote a copy of $\omega$ with the reversed order.

    \begin{proposition}\label{div_fin}
		Let $A$ be a countable set of parameters and $\varphi(x,b)$ a formula which divides over $A$. Then $\varphi(x,b)$ divides finitely over $A$ (so it divides uniformly over $A$ by \cite[Lemma 12.3.9]{Granger}).
		
		Moreover, if $k$ is such that $\varphi(x,b)$ $k$-divides over $A$, then there is an $A$-indiscernible sequence witnessing this contained in a model of $T_{2k|V(b)|}$.
	\end{proposition}
\begin{proof}
    Let $k$ be such that $\varphi(x,b)$ $k$-divides over $A$.
	By compactness we can find an $A$-indiscernible sequence $(b_i)_{i<\omega\frown \omega^*}$ with $b_{0^*}=b$ such that the set $\{\varphi(x,b_i):i<\omega\frown \omega^*\}$ is $k$-inconsistent. Note that for any $i$ the type $tp(b_{i^*}/Ab_{<\omega}b_{(<i)^*})$ is finitely satisfiable in $b_{<\omega}$, hence in particular, putting $C:=Ab_{<\omega}$ we get that the sequence $I:=(b_{i^*})_{i<\omega}$ is Morley over $C$. Hence $\varphi(x,b)$ Kim-divides over $C$, so (by Kim's Lemma), it $\Gamma$MS-divides over $C$. Thus, by \cite[Lemma 12.3.12]{Granger}, $\varphi(x,b)$ divides finitely over $C$ so, as $A\subseteq C$, it divides finitely over $A$.
	
	Moreover, the proof \cite[Lemma 12.3.12]{Granger} gives that a $C$-indiscernible sequence witnessing $k$-dividing of $\varphi(x,b)$ over $C$ can be found in an $R$-dimensional subspace of $V(\FC)$ provided that there is a model $N_R\models T_R$ containing $k$ elements of a $\Gamma$-Morley sequence witnessing $\Gamma$MS-dividing of $\varphi(x,b)$ over $C$. But such $k$ elements contain at most $k|V(b)|$ vectors, so there is a model of $T_{2k|V(b)|}$ containing all of them (see e.g. \cite[Fact 3.2]{bilinear}). This gives the `moreover' clause.
\end{proof}
By \cite[Remark 12.3.4]{Granger} we conclude:
\begin{corollary}
	In $T_\infty$, forking and dividing coincide: for any small set $A$ and any formula $\varphi(x,b)$ we have that $\varphi(x,b)$ forks over $A$ if and only if $\varphi(x,b)$ divides over $A$.
\end{corollary}

Next, we will describe dividing of formulae in $T_\infty$ in terms of dividing in the $\omega$-stable theories $T_m$ (which are interpretable in $ACF$). We will need to recall some definitions and facts about approximating sequences.

By quantifier elimination in $T_\infty$   (see \cite[9.2.3]{Granger}, \cite[Fact 2.8]{bilinear}), we may (and will) restrict ourselves to considering  only  quantifier-free formulae.

\begin{fact}\label{cons}
Let $M\models T_m$ (with $M\subseteq \FC\models T_\infty$) and let $\varphi(x,a)$ be a (quantifier-free) formula with $a\subseteq M$. Then if $\varphi(x,a)$ is consistent in $T_\infty$ and $m\geq 2(|V(x)|+|V(a)|)$, then $\varphi(x,a)$ is consistent in $M$. 
\end{fact}
\begin{proof}
Suppose $\FC\models \varphi(b,a)$ for some $b\subseteq \FC$. Then, as $m\geq 2|V(ab)|$, there is a model $N\models T_m$ containing $ba$ (see \cite[Fact 3.2]{bilinear}). As $\varphi(x,a)$ is quantifier-free, we have $N\models \varphi(b,a)$. By quantifier elimination in $T_m$ we conclude that $\varphi(x,a)$ is consistent in $M$.
\end{proof}

\begin{proposition}
Suppose $A\subseteq \FC$,  $|A_V|<\omega$, $\varphi(x,b)$ is a (quantifier-free) formula, and $k<\omega$.  Let $M\subseteq \FC$ be a model of $T_{m}$ with $m\geq 2k|V(b)|$ containing $A$ and $b$. Then $\varphi(x,b)$ $k$-divides over $A$ in $T_\infty$ if and only if it $k$-divides over $A$ in $M$.
\end{proposition}
\begin{proof}
Suppose first that $\varphi(x,b)$ $k$-divides over $A$ in $T_\infty$. By Proposition \ref{div_fin}, we can find a sequence $(b_i)_{i<\omega}$ contained in a model  $N\models T_{m}$ witnessing this. Then, as  $\{\varphi(x,b_i):i<\omega\}$ is $k$-inconsistent in $T_\infty$, it is $k$-inconsistent in $N$ as well. Hence, by quantifier elimination in $T_m$, $\varphi(x,b)$ $k$-divides in $M$.

Conversely, suppose $\varphi(x,b)$ $k$-divides in $M$, and consider arbitrary $l<\omega$ and a finite subset $p_0\subseteq \tp(b/A)$. Then there is a sequence $(b_i)_{i<l}$ of tuples in $M$ such that for any $i_1<\dots<i_k$ we have that $\bigwedge_{j<k} \varphi(x,b_{i_j})$ is inconsistent in $M$ and $b_i\models p_0$ for every $i<l$. Then, as $b_{i_0}\dots b_{i_{k-1}}$ contains at most $k|V(b)|$ vectors and $m\geq 2k|V(b)|$, we get by Fact \ref{cons} that $\bigwedge_{j<k} \varphi(x,b_{i_j})$ is inconsistent in $T_\infty$ as well. So, by compactness, $\varphi(x,b)$ $k$-divides in $T_\infty$. 
\end{proof}

\section{Further examples}\label{sec:examples}
In this section, we collect more examples of NSOP$_1$ theories with existence.
In each of them, there exist at least three related notions of independence: $\ind$, $\ind^K$ and $\ind^{K,q}$
(recall that $\ind^{K,q}$ denotes Kim-independence with respect to definition (B) of Kim-dividing).
We would like to better understand what are the relations between these three notions of independence.
Obviously: $\ind\,\Rightarrow\,\ind^K\,\Rightarrow\,\ind^{K,q}$.
The first natural question is: when $\ind^K$ and $\ind^{K,q}$ coincide?

\begin{remark}
Let $T$ be NSOP$_1$ with existence. Then $\ind^K =\ind^{K,q}$ if and only if for every $C$ and every $p(x)\in S(C)$ we have that $p(x)$ extends to a global $C$-invariant type.
The implication ``$\Rightarrow$'' holds for arbitrary $T$.
\end{remark}
\begin{proof}
``$\Leftarrow$'':  We only need to show that $a\ind^{K,q}_C B \Rightarrow a\ind^K_CB$ for arbitrary $a,B,C$. By assumption $\tp(B/C)$ extends to a global $A$-invariant type $r(x)$. If  $a\ind^{K,q}_C$, then there is an $aC$-indiscernible Morley sequence in $r(x)$ over $A$ starting with $C$. As $I$ is in particular $\ind$-Morley over $A$, by Kim's Lemma we get that $a\ind^K_C B$.\\
``$\Rightarrow$'': Suppose there exist $B$ and $C$ such that $\tp(B/C)$ does not extend to a global $C$-invariant type. Then for any $D$ we have that $\tp(BD/C)$ does not extend to a global $C$-invariant type, so, vacuously, $E\ind^{K,q}_C BD$ for any $D,E$.
But we can choose (assuming $T$ has infinite models) $D$ and $E$ such that $E\subseteq \acl(D)\setminus \acl(C)$, in which case $E\nind^K_C BD$, hence $\ind^{K,q}$ does not imply $\ind^K$ in $T$.
\end{proof}

\begin{remark}\label{rem:two.K.ind}
\begin{enumerate}
    \item Note in particular that if there exists $C$ with $\acl(C)\neq \dcl(C)$ then for any $c\in \acl(C)\setminus \dcl(C)$ the type $\tp(c/C)$ does not extend to a global $C$-invariant type, so $\ind^K \neq \ind^{K,q}$.
    
    \item By exactly the same argument as above, if $T$ is NSOP$_1$ with existence, then $\ind^{K,q}$ coincides with $\ind^K$ over algebraically closed sets if and only if for any algebraically closed $C$ any $p(x)\in S(C)$ we have that $p(x)$ extends to a global $C$-invariant type. In particular, if $T=T^{eq}$ is stable then $\ind^{K,q}$ coincides with $\ind^K=\ind$ over algebraically closed sets.
\end{enumerate}
\end{remark}

\begin{example}
We are working in $T_{\infty}$.
As there are sets $C$ with $\dcl(C)\neq\acl(C)$ (for example in the sort $K(\FC)$, on which the induced structure is just that of a pure algebraically closed field), we have that $\ind^K\neq \ind^{K,q}$ in $T_\infty$.
We expect that $\ind^K=\ind^{K,q}$ over algebraically closed sets
which however relies on Question \ref{acl_Granger} below: by the proof of \cite[Proposition 8.1]{bilinear}, any complete type over a finite set $C$ extends to a global type $p(x)$ which is  invariant over $\acl^{eq}(C)$. By compactness, the same is true for arbitrary $C$, as the existence of a global $\acl^{eq}(C)$-invariant extension of $p(x)$ is equivalent to consistency of the type $p(x)\cup \{\varphi(x,d)\leftrightarrow \varphi(x,d'):\varphi(x,y)\in L, d,d'\in \FC^{eq}, d\equiv_{\acl^{eq}(C)}d'\}$.
So, if $\dcl^{eq}(\acl(C))=\acl^{eq}(C)$ then $p(x)$ extends to a global $\acl(C)$ -invariant type in $T^{eq}_\infty$, which restricts to a global $\acl(C)$-invariant type in $T_\infty$.
\end{example}

\begin{question}\label{acl_Granger}
Is it true in $T_\infty$ that for any set $C$ we have $\dcl^{eq}(\acl(C))=\acl^{eq}(C)$?
\end{question}

\begin{example}
Assume that $T$ is the theory of $\omega$-free PAC fields,
let $F^*$ be a monster model of $T$ 
and let $\bar{F}:=(F^*)^{\sep}$ (so $\bar{F}$ is a model of $\scf$ and in this example, if we refer to $\scf$, we actually refer to $\theo({\bar{F}})$).
Assume that $A=\acl(A)$, $B=\acl(B)$, $C=\acl(C)$
(by \cite{ChaPil}, we know that $\acl(\ast)$ is obtained by closing under $\lambda$-functions and then taking the field-theoretic algebraic closure) and $C\subseteq A\cap B$.
By \cite{zoe2002}, we have $A\ind_C B$ if and only if
\begin{itemize}
    \item $A\ind_C^{\scf}B$ (forking independence in $\scf$) and
    \item $\acl(AB)\cap (AB_0)^{\sep}B^{\sep}=\acl(AB_0)B$ for each $B_0=\acl(B_0)\subseteq B$.
\end{itemize}
Let $A'$ be a small subset of $F^*$ and let $B=\acl(B)$. By the general properties of  forking independence (e.g. Remark 5.3 in \cite{casasimpl}), $A'\ind_B B$ holds if and only if $A\ind_B B$, where $A:=\acl(A'B)$.
By the above description of forking independence in the theory of $\omega$-free PAC fields, we have $A'\ind_B B$ (and so also $A'\ind^K_B B$ and $A'\ind^{K,q}_B B$).
So we see that the existence axiom for forking independence holds over algebraically closed sets.
Actually, the existence axiom for forking independence holds over arbitrary sets (cf. Remark 2.15 and Remark 2.16 in \cite{DKR}).

Let us recall what is $\ind^K$ and $\ind^{K,q}$ in the case of $\omega$-free PAC fields.
Up to our knowledge, there is no description of $\ind^K$ or $\ind^{K,q}$ over arbitrary sets, so we need to pick-up a model $F\preceq F^*$. Assume that $A=\acl(A)$ and $B=\acl(B)$, and $F\subseteq A\cap B$. Then, $A\ind^K_F B$ if and only if $A\ind^{K,q}_F B$ (as both notions of Kim-dividing coincide over models in NSOP$_1$), if and only if
\begin{itemize}
    \item $A\ind^{\scf}_F B$ and
    \item $S\mathcal{G}(A)\ind^K_{S\mathcal{G}(F)}S\mathcal{G}(B)$,
\end{itemize}
where (2) is considered in the so-called sorted system,
which is a first order structure build from all the quotients by open normal subgroups of the absolute Galois group
(e.g. see \cite{cherlindriesmacintyre}, \cite{zoeIP}, \cite{HL1}).
In fact, as the absolute Galois group is the free profinite group on $\omega$-many generators, this sorted system is stable (by \cite{zoeIP}) and hence the second dot above can be replaced by ``$S\mathcal{G}(A)\ind_{S\mathcal{G}(F)}S\mathcal{G}(B)$", i.e. the forking independence relation, which is
described in Proposition 4.1 in \cite{zoeIP}.
\end{example}

\begin{example}
Consider the theory ACFG, exposed in \cite{dElbee21} and put it in the place of $T$.
More precisely, fix some prime number $p>0$ and let $L_G$ be the language of rings extended by a symbol for a unary predicate $G$. Models of the $L_G$-theory ACF$_G$ are algebraically closed fields of characteristic $p$ with an additive subgroup under the predicate $G$. 
Now, the $L_G$-theory ACFG is the model companion of ACF$_G$; let $(K,G)$ be its monster model.
For $A\subseteq K$, let $\bar{A}$ denote the field-theoretic algebraic closure of $A$ in $K$.
Take $A,B,C\subseteq K$ and recall that the \emph{weak independence relation}, $A\ind^w_C B$, is given as follows:
\begin{itemize}
    \item $A\ind^{\text{ACF}}_C B$ and
    \item $G(\overline{AC}+\overline{BC})=G(\overline{AC})+G(\overline{BC}),$
\end{itemize}
where $\ind^{\text{ACF}}$ denotes the forking independence relation in ACF.
By Corollary 3.16 from \cite{dElbee21}, we have that $\ind$ is $\ind^w$ after forcing base monotonicity:
$$A\ind_C B\quad\iff\quad (\forall D\subseteq CB)\big(A\ind^w_{CD}BC \big).$$
Therefore ACFG, which is a NSOP$_1$ theory, enjoys the existence axiom for forking independence.
In a private communication with Christian D’Elb\'{e}e, it was suggested to us to use 
Corollary 1.7 from \cite{dElbee21} and Theorem \ref{Kim.lemma.arbitrary}, to show that $\ind^w$ coincides with $\ind^K$ over algebraically closed sets.
As the proof is standard, we omit it here. Therefore $\ind^K$ coincides with $\ind^w$ over algebraically closed subsets and
$\ind^{K,q}$ coincides with $\ind^w$ over models.
In ACFG, there exist sets $C$ such that $\dcl(C)\neq\acl(C)$, thus $\ind^K\neq\ind^{K,q}$ in general (by Remark \ref{rem:two.K.ind}). 
One could ask whether $\ind^K$ and $\ind^{K,q}$ coincide over algebraically closed sets (similarly as in $T_{\infty}$)?
\end{example}

\begin{example}
In \cite{conant_kruckman_2019}, another example of an  NSOP$_1$ theory with existence was studied.
More precisely, $T_{m,n}$ denotes the theory of existentially closed incidence structures omitting the complete incidence structure $K_{m,n}$. In other words, $T_{m,n}$ is the theory of a generic $K_{m,n}$-free bipartite graph.
Theorem 4.11 from \cite{conant_kruckman_2019} gives us that $T_{m,n}$ is NSOP$_1$ and Corollary 4.24 from \cite{conant_kruckman_2019} implies that $T_{m,n}$ satisfies the existence axiom for  forking independence.
Actually, the aforementioned corollary gives us also a nice description of forking independence, thus let us evoke it: for any $A,B,C\subseteq\FC\models T_{m,n}$ we have
$$A\ind_C B\quad\iff\quad A\ind^d_C B\quad\iff\quad (\forall D\subseteq\acl(BC),\, C\subseteq D)(A\ind^I_D B),$$
and $\ind^I$ is a ternary relation coinciding with $\ind^{K,q}$ over models.
Thus, again, forking independence is induced by forcing base monotonicity on a ternary relation related to $\ind^{K,q}$.
\end{example}

The above examples motivate asking the following question:

\begin{question}\label{question1}
Assume that $T$ is NSOP$_1$ with existence. Is it true that $\ind$ is $\ind^K$ after forcing base monotonicity, i.e.:
$$A\ind_C B\quad\iff\quad (\forall D\subseteq\acl(BC),\, C\subseteq D)(A\ind^K_D B)\,?$$
\end{question}

We know that in any theory $T$ (not necessarily an  NSOP$_1$ theory), for each $a$ and $A$ it holds that $a\ind^{K,q}_A A$\footnote{We proved this fact under assumption of NSOP$_1$ with the use of our rank, but after a talk on the topic, Itay Kaplan, in a private communication, shared with us a simpler proof of this fact, which does not need the NSOP$_1$ assumption.}. 
Assuming that Question \ref{question1} has an affirmative answer and assuming that 
$\ind^K$ coincide with $\ind^{K,q}$ over algebraically closed sets, we would obtain the forking existence axiom over algebraically closed sets.

\section{Beyond NSOP$_1$}\label{sec:beyond}
Now, let us consider arbitrary theory $T$ (not necessarily NSOP$_1$).
Instead of Definition \ref{def:Qrank}, one could define the rank by the conditions from Lemma \ref{lemma:technical}. Such a rank would be always finite and satisfy the standard properties (Lemma \ref{lemma:implies}, Lemma \ref{lemma:Qrank.vee}, Lemma \ref{lemma:finite.character}, Corollary \ref{cor:extensions}). This is some strategy, however it seems that we require too much here. Thus, let us derive yet another notion of rank and show its finiteness in an important class of theories.
The new rank is a slight modification of Q-rank and, as we will see in a moment, the new rank coincides with the Q-rank in NSOP$_1$ theories.

\subsection{Refining notions}
We start with introducing a refined notion of an invariant type and a notion of a surrogate of a global type.
These notions are used in Definition \ref{def:Qrank2} and came out from studying examples similar to the one  in Subsection \ref{ssec:exmample}.

\begin{definition}\label{def:new.inv}
Let $A,B$ be small subsets of $\FC$ and let $C\subseteq\FC$.
A type $q(y)\in S(C)$ is $B/A$-invariant if $f(q)=q$ for every $f\in\aut(\FC/A)$ such that $f(B)=B$.
(This implies some restrictions on the set of parameters $C$.)
\end{definition}

\begin{remark}
Let $A\subseteq B$ be small subsets of $\FC$, let $C\subseteq\FC$ and let $q(y)\in S(C)$.
We have:
$q$ is $A$-invariant $\Rightarrow$ $q$ is $B/A$-invariant $\Rightarrow$ $q$ is $B$-invariant.
\end{remark}

\begin{definition}
A type $q(y)$ is a semi-global type over $A$ if
there exists $\pi(y)\in S(A)$ such that $q(y)\in S(A\,\cup\,\pi(\FC))$ and $\pi\subseteq q$.
\end{definition}

\begin{definition}[Morley in semi-global type]
Let $q$ be an $A$-invariant semi-global type over $A$ and let $(I,<)$ be a linearly ordered set.
By a \emph{Morley sequence in $q$ over $A$ over $A$ (of order type $I$)} we understand a sequence
$\bar{B}=(b_i)_{i\in I}$ such that $b_i\models q|_{Ab_{<i}}$.
\end{definition}

\begin{remark}
Assume that $A\subseteq B$ are small subsets of $\FC$,
and $q$ is an $B/A$-invariant semi-global type over $B$ and $(I,<)$ is linearly ordered set.
\begin{enumerate}
    \item 
    For each $f\in\aut(\FC/A)$ with $f(B)=B$,
    we have $f(q)=q\in S(B\,\cup\,\pi(\FC))$ for some $\pi(y)\in S(B)$.
    More precisely, $f(\pi)=\pi\in S(B)$ ($B$ is fixed setwise), so
    $f(B\,\cup\,\pi(\FC))=B\,\cup\,\pi(\FC)$, and so we do not need to add new parameters to the domain of $q$
    (to satisfy $B/A$-invariance of $q$) and being a $B/A$-invariant semi-global type over $B$ is well-defined.
    
    \item
    If $b_I=(b_i)_{i\in I}$ and $c_I=(c_i)_{i\in I}$ are Morley sequences in $q$ over $B$, then $b_I\equiv_B c_I$.
    Thus we see that $b_I$ and $c_I$ realize a common type over $B$, which we denote $q^{\otimes I}|_B$.
    Moreover, $b_I$ is $B$-indiscernible and $b_i\ind^d_B b_{<i}$ for all $i\in I$.
\end{enumerate}
\end{remark}

\begin{lemma}
Assume that that $q$ and $r$ are $A$-invariant semi-global types over $A$, $tp(b/A)=r|_A=q|_A$.
If for a formula $\varphi(x,b)$ there is $b_{<\omega}\models q^{\otimes\omega}|_A$ and $c_{<\omega}\models r^{\otimes\omega}|_A$ such that $\{\varphi(x,b_i)\;|\;i<\omega\}$ is consistent and $\{\varphi(x,c_i)\;|\;i<\omega\}$ is inconsistent, then $T$ has SOP$_1$.
\end{lemma}

\begin{proof}
As $tp(b/A)=r|_A=q|_A$, we have that $q,r\in S(A\,\cup\,\tp(b/A)(\FC))$ and we can repeat the proof of Proposition 3.15 from \cite{kaplanramsey2017}.
\end{proof}

\begin{corollary}\label{cor:Kim.lemma.2}
    If $T$ is NSOP$_1$ then we have \emph{Kim's lemma for semi-global types dividing}:
    for any small set $A\subseteq\FC$, any formula $\varphi(x,b)$, if there is an $A$-invariant semi-global type $q\supseteq\tp(b/A)$ over $A$
    and a sequence $b_{<\omega}\models q^{\otimes\omega}|_A$ such that $\{\varphi(x,b_i)\;|\;i<\omega\}$ is inconsistent,
    then $\{\varphi(x,c_i)\;|\;i<\omega\}$ is inconsistent for any $A$-invariant semi-global type $r\supseteq\tp(b/A)$ over $A$ and any sequence $c_{<\omega}\models r^{\otimes\omega}|_A$.
\end{corollary}

\begin{proposition}
The following are equivalent for the complete theory $T$:
\begin{enumerate}
    \item $T$ is NSOP$_1$.
    
    \item (Kim's lemma for semi-global types dividing)
    For any small $M\preceq\FC$, any formula $\varphi(x,b)$, if there is an $M$-invariant semi-global type $q\supseteq\tp(b/M)$ over $M$
    and a sequence $b_{<\omega}\models q^{\otimes\omega}|_M$ such that $\{\varphi(x,b_i)\;|\;i<\omega\}$ is inconsistent,
    then $\{\varphi(x,c_i)\;|\;i<\omega\}$ is inconsistent for any $M$-invariant semi-global type $r\supseteq\tp(b/M)$ over $M$ and any sequence $c_{<\omega}\models r^{\otimes\omega}|_M$.
    
    \item (Kim's lemma for Kim-dividing)
    For any small $M\preceq\FC$, any formula $\varphi(x,b)$, if there is an $M$-invariant global type $q\supseteq\tp(b/M)$
    and a sequence $b_{<\omega}\models q^{\otimes\omega}|_M$ such that $\{\varphi(x,b_i)\;|\;i<\omega\}$ is inconsistent,
    then $\{\varphi(x,c_i)\;|\;i<\omega\}$ is inconsistent for any $M$-invariant global type $r\supseteq\tp(b/M)$ and any sequence $c_{<\omega}\models r^{\otimes\omega}|_M$.
\end{enumerate}
\end{proposition}

\begin{proof}
(1)$\Rightarrow$(2) follows by Corollary \ref{cor:Kim.lemma.2}.
If $q$ and $r$ are $M$-invariant global types extending $\tp(b/M)$
then $q|_B$ and $r|_B$, where $B:=M\,\cup\,\tp(b/M)(\FC)$, are $M$-invariant semi-global types over $M$ and thus (3) is implied by (2). Finally, (3)$\Rightarrow$(1) is contained in Theorem 3.16 from \cite{kaplanramsey2017}.
\end{proof}

\begin{remark}
Say that we are interested in the notion of Kim-dividing (over $A$) of a formula $\varphi(x,b)$.
As it was earlier pointed out, the main issue with the definition of Kim-dividing from \cite{kaplanramsey2017}
is that there might be not enough many $A$-invariant global extensions of $\tp(b/A)$ to witness Kim-dividing.
Perhaps there are situations, where it is easier to find an $A$-invariant semi-global extension (over $A$) of $\tp(b/A)$ and so working with semi-global types (instead of global types) in the definition of Kim-dividing from \cite{kaplanramsey2017} might give better results.
\end{remark}

\subsection{Modifying rank}
Again, let 
$$Q:=\big((\varphi_0(x;y_0),q_0(y_0)),\ldots,(\varphi_{n-1}(x;y_{n-1}),q_{n-1}(y_{n-1}) \big),$$
where $\varphi_0,\ldots,\varphi_{n-1}\in\CL$ and $q_0,\ldots,q_{n-1}$ are global types.

\begin{definition}\label{def:Qrank2}
We define a local rank, called \emph{\~{Q}-rank}, 
$$\tilde{D}_Q(\,\cdot\,):\big\{\text{sets of formulae}\}\to\Ord\cup\{\infty\}.$$
For any set of $\mathcal{L}$-formulae $\pi(x)$
we have $\tilde{D}_Q(\pi(x))\geqslant\lambda$ if and only if there exists 
$\eta\in n^{\lambda}$ and 
$\big(b^{\alpha}_{\Zz}=(b^{\alpha}_i)_{i\in\Zz},M^{\alpha}\big)_{\alpha<\lambda}$ such that
\begin{enumerate}
    \item $\dom(\pi(x))\subseteq M^0$,
    \item $q_0,\ldots,q_{n-1}$ are $M^0$-invariant,
    \item $M^{\alpha}\preceq\FC$ for each $\alpha<\lambda$, $(M^{\alpha})_{\alpha<\lambda}$ is continuous, and each $M^{\alpha+1}$ is $|M^{\alpha}|^+$-saturated and strongly $|M^{\alpha}|^+$-homogeneous,
    \item $b^{\alpha}_{\Zz}M^{\alpha}\subseteq M^{\alpha+1}$ for each $\alpha+1<\lambda$,
    \item $b^{\alpha}_0\models q_{\eta(\alpha)}|_{M^{\alpha}}$ for each $\alpha<\lambda$,    
    \item $\pi(x)\cup\{\varphi_{\eta(\alpha)}(x;b^{\alpha}_0)\;|\;\alpha<\lambda\}$ is consistent,

    \item for each $\alpha<\lambda$ there exists an $M^{\alpha}/M^0$-invariant semi-global type $r_{\alpha}(y_{\eta(\alpha)})$
    over $M^{\alpha}$ such that 
    $b^{\alpha}_{\Zz}\models r_{\alpha}^{\otimes\Zz}|_{M^{\alpha}}$ and $\{\varphi_{\eta(\alpha)}(x;b^{\alpha}_i)\;|\;i\in\Zz\}$ is inconsistent.
\end{enumerate}
If  $\tilde{D}_Q(\pi)\geqslant\lambda$ for each $\lambda\in\Ord$,
then we set $\tilde{D}_Q(\pi)=\infty$. 
Otherwise $\tilde{D}_Q(\pi)$ is the maximal $\lambda\in\Ord$ such that $\tilde{D}_Q(\pi)\geqslant\lambda$.
\end{definition}

\begin{proposition}
If $T$ is NSOP$_1$ then $D_Q=\tilde{D}_Q$.
\end{proposition}

\begin{proof}
If $D_Q(\pi)\geqslant\lambda$ then 
there exist $\eta\in n^{\lambda}$ and $(b^{\alpha},M^{\alpha})_{\alpha<\lambda}$ as in Definition \ref{def:Qrank}.
By Remark \ref{rem:cond.3}, we may assume that each $M^{\alpha+1}$ is $|M^{\alpha}|^+$-saturated and strongly $|M^{\alpha}|^+$-homogeneous. Then, as in Lemma \ref{lemma:Qrank.NSOP1}, we may assume that each $r_{\alpha}\in S(\FC)$ from the condition (7) of Definition \ref{def:Qrank} is equal to $q_{\eta(\alpha)}$ and that we have $b^{\alpha}_{<\omega}\subseteq M^{\alpha+1}$ satisfying $q_{\eta(\alpha)}^{\otimes\omega}|_{M^{\alpha}}$
and such that $\{\varphi_{\eta(\alpha)}(x,b^{\alpha}_i)\;|\;i<\omega\}$ is inconsistent.
By saturation of $M^{\alpha+1}$, we can extend each $b^{\alpha}_{<\omega}$ to $b^{\alpha}_{\Zz}\models q_{\eta(\alpha)}^{\otimes\Zz}|_{M^{\alpha}}$ contained in $M^{\alpha+1}$ such that
$\{\varphi_{\eta(\alpha)}(x,b^{\alpha}_i)\;|\;i\in\Zz\}$ is inconsistent. 
Finally, $q_{\eta(\alpha)}|_{B^{\alpha}}$, where $B^{\alpha}:=M^{\alpha}\,\cup\,q_{\eta(\alpha)}|_{M^{\alpha}}(\FC)$,
is an $M^{\alpha}/M^0$-invariant semi-global type over $M^{\alpha}$. Therefore we fulfill all the conditions of Definition \ref{def:Qrank2} and we see that $\tilde{D}_Q(\pi)\geqslant\lambda$.

Now, assume that $\tilde{D}_Q(\pi)\geqslant\lambda$. 
There exists $\eta\in n^{\lambda}$ and 
$\big(b^{\alpha}_{\Zz}=(b^{\alpha}_i)_{i\in\Zz},M^{\alpha}\big)_{\alpha<\lambda}$ as in Definition \ref{def:Qrank2}.
All the first six conditions from Definition \ref{def:Qrank} are naturally satisfied by $\eta$, $M^{\alpha}$ and $b^{\alpha}_{<\omega}$, where $\alpha<\lambda$.
To see that also the condition (7) of Definition \ref{def:Qrank} is satisfied, we argue as follows.
Let $\alpha<\lambda$, let $c^{\alpha}_{<\omega}\models q_{\eta(\alpha)}^{\otimes\omega}|_{M^{\alpha}}$ be contained in $M^{\alpha+1}$ and such that $c^{\alpha}_0=b^{\alpha}_0$.
We know that $q_{\eta(\alpha)}$ is $M^0$-invariant, so in particular also $M^{\alpha}$-invariant.
Moreover, as 
 $q_{\eta(\alpha)}|_{B^{\alpha}}$, where $B^{\alpha}:=M^{\alpha}\,\cup\,q_{\eta(\alpha)}|_{M^{\alpha}}(\FC)$,
is an $M^{\alpha}$-invariant semi-global type over $M^{\alpha}$, and 
there exists an $M^{\alpha}/M^0$-invariant semi-global type $r_{\alpha}(y_{\eta(\alpha)})$
    over $M^{\alpha}$ such that 
    $b^{\alpha}_{\Zz}\models r_{\alpha}^{\otimes\Zz}|_{M^{\alpha}}$ and we have that $\{\varphi_{\eta(\alpha)}(x;b^{\alpha}_i)\;|\;i\in\Zz\}$ is inconsistent (by the condition (7) of Definition \ref{def:Qrank2}),
Corollary \ref{cor:Kim.lemma.2} implies that 
$\{\varphi_{\eta(\alpha)}(x;c^{\alpha}_i)\;|\;i\in\Zz\}$ is inconsistent.
\end{proof}

\begin{corollary}
    If $T$ is NSOP$_1$ then $\tilde{D}_Q(\pi)<\omega$ for any choice of $\pi$ and finite $Q$.
\end{corollary}

The following easily follows by the definition.

\begin{fact}\label{fact:2}
\begin{enumerate}
\item 
$\tilde{D}_Q(\pi))\geqslant\lambda$, $f\in\aut(\FC)$ $\Rightarrow$ $\tilde{D}_{f(Q)}(f(\pi))\geqslant\lambda$.

%\item 
%If $\pi'\subseteq \pi$ then $\tilde{D}_Q(\pi)\leqslant \tilde{D}_{Q}(\pi')$.

\item 
If $\dom(\pi')\subseteq\dom(\pi)$ and $\pi\vdash\pi'$ then $\tilde{D}_Q(\pi)\leqslant\tilde{D}_Q(\pi')$.

\item
$\tilde{D}_Q(\pi)\leqslant\sum\limits_{j<n}\tilde{D}_{((\varphi_j,q_j))}(\pi)$.

\item
We have
$$\tilde{D}_Q(\pi\,\cup\,\{\bigvee\limits_{j\leqslant m}\psi_j\})\leqslant\max\limits_{j\leqslant m}\tilde{D}_Q(\pi\,\cup\,\{\psi_j\}).$$

\end{enumerate}
\end{fact}

\begin{question}
Working outside of NSOP$_1$, can we improve statements (2) and (4) in Fact \ref{fact:2}
so they will get theses as in Lemma \ref{lemma:implies} and Lemma \ref{lemma:Qrank.vee}?
\end{question}

\begin{proposition}\label{prop:tilde.rank}
If $\tilde{D}_Q(\pi(x))\geqslant\lambda$ then there exist
natural numbers $l_{\alpha}$, formulae $\psi_{\alpha}(x,y_{\alpha})$ and sequences $c^{\alpha}_{\Zz}:=(c^{\alpha}_i)_{i\in\Zz}$, where $\alpha<\lambda$ such that
\begin{itemize}
    \item for each $\sigma:\lambda\to\Zz$ the set
    $$\pi(x)\,\cup\,\bigcup\limits_{\alpha<\lambda}\{\psi_{\alpha}(x,c^{\alpha}_{\sigma(\alpha)}),\,\neg\psi_{\alpha}(x,c^{\alpha}_l)\;|\;l\neq\sigma(\alpha)\}$$
    is consistent,
    
    \item for each $\alpha<\lambda$ the set
    $$\{\psi_{\alpha}(x,c^{\alpha}_l)\;|\;l\in\Zz\}$$
    is $l_{\alpha}$-inconsistent.
\end{itemize}
\end{proposition}

\begin{proof}
Consider $((b^{\alpha}_{\Zz}),M^{\alpha})_{\alpha<\lambda}$ as in Definition \ref{def:Qrank2}.
For simpler notation, let us introduce small $M^{\lambda}\preceq\FC$ which contains all $M^{\alpha}$'s, all $b^{\alpha}_{\Zz}$'s and is $|\bigcup_{\alpha}M^{\alpha}|^+$-saturated and strongly $|\bigcup_{\alpha}M^{\alpha}|^+$-homogeneous.
We know that $\pi(x)\,\cup\,\{\varphi_{\eta(\alpha)}(x,b^{\alpha}_0)\;|\;\alpha<\lambda\}$ is consistent.
\ 
\\
\\
\textbf{Claim 1:}
There are natural numbers $(k_{\alpha})_{\alpha<\lambda}$ such that
$$\pi(x)\,\cup\,\bigcup\limits_{\alpha<\lambda}\Big(\{\varphi_{\eta(\alpha)}(x,b^{\alpha}_l)\;|\;0\leqslant l<k_{\alpha}\}\,\cup\,\{\neg\varphi_{\eta(\alpha)}(x,b^{\alpha}_l)\;|\;l<0\text{ or }l\geqslant k_{\alpha}\}\Big)$$
is consistent.
\ 
\\
\textit{Proof of Claim 1:} We will choose $k_{\alpha}$'s recursively.
Let $k_0<\omega$ be maximal such that
$$\pi(x)\,\cup\,\{\varphi_{\eta(0)}(x,b^{0}_l)\;|\;0\leqslant l<k_0\}\,\cup\,\{\varphi_{\eta(\alpha)}(x,b^{\alpha}_0)\;|\;0<\alpha<\lambda\}$$
is consistent. Then also
\begin{IEEEeqnarray*}{rCl}
\pi(x) & \cup & \{\varphi_{\eta(0)}(x,b^{0}_l)\;|\;0\leqslant l<k_0\} \\
&\cup & \{\neg\varphi_{\eta(0)}(x,b^{0}_l)\;|\;l<0\text{ or }l\geqslant k_{0}\} \\
& \cup & \{\varphi_{\eta(\alpha)}(x,b^{\alpha}_0)\;|\;0<\alpha<\lambda\}
\end{IEEEeqnarray*}
is consistent. If not then, by compactness, 
\begin{equation}\label{eq:6.1}
\pi(x)\,\cup\,\{\varphi_{\eta(0)}(x,b^{0}_l)\;|\;0\leqslant l<k_0\}\,\cup\,\{\varphi_{\eta(0)}(x,b^0_s)\}\,\cup\,\{\varphi_{\eta(\alpha)}(x,b^{\alpha}_0)\;|\;0<\alpha<\lambda\}
\end{equation}
is consistent for some $s<0$ or $s\geqslant k_0$.
Assume $s<0$. Because $b^0_s b^0_0\ldots b^0_{k_0-1}\equiv_{M^0}b^0_0\ldots b^0_{k_0}$ and
$b^0_{\Zz}\subseteq M^1$, and $M^1$ is $|M^0|^+$-saturated and strongly $|M^0|^+$-homogeneous,
there exists
$f_1\in\aut(M^1/M^0)$ such that 
such that $f^0_1(b^0_s b^0_0\ldots b^0_{k_0-1})=b^0_0\ldots b^0_{k_0}$. 

We recursively construct an increasing sequence $f_{\alpha}\in\aut(M^{\alpha}/M^0)$ for $1\leqslant\alpha\leqslant\lambda$. 
Assume that we have $f_{\alpha}\in\aut(M^{\alpha}/M^0)$ and we need to define $f_{\alpha+1}$.
Consider any extension $f_{\alpha}\subseteq f_{\alpha}'\in\aut(M^{\alpha+1}/M^0)$.
As $b^{\alpha}_{\Zz}\models r_{\alpha}^{\otimes\Zz}|_{M^{\alpha}}$ 
($r_{\alpha}$ as in the condition (7) of Definition \ref{def:Qrank2}), 
we have $f'_{\alpha}(b^{\alpha}_{\Zz})\models f'_{\alpha}(r_{\alpha})^{\otimes\Zz}|_{f'_{\alpha}(M^{\alpha})}$.
We know that $f'_{\alpha}(M^{\alpha})=M^{\alpha}$, $f'_{\alpha}|_{M^0}=\id_{M^0}$ and that $r_{\alpha}$
is $M^{\alpha}/M^0$-invariant. Thus $f'_{\alpha}(b^{\alpha}_{\Zz})\models r_{\alpha}^{\otimes\Zz}|_{M^{\alpha}}$
and $f'_{\alpha}(b^{\alpha}_{\Zz})\equiv_{M^{\alpha}}b^{\alpha}_{\Zz}$.
There exists $f''_{\alpha}\in\aut(M^{\alpha+1}/M^{\alpha})$ such that $f''_{\alpha}f'_{\alpha}(b^{\alpha}_{\Zz})=b^{\alpha}_{\Zz}$.
Set $f_{\alpha+1}:=f''_{\alpha}\circ f'_{\alpha}\in\aut(M^{\alpha+1}/M^0)$ and note that $f_{\alpha}\subseteq f_{\alpha+1}$. If $\delta\leqslant\lambda$ is a limit ordinal, then simply put  $f_{\delta}:=\bigcup\limits_{1\leqslant\alpha<\delta}f_{\alpha}\in\aut(M^{\delta}/M^0)$.

%Now, let $f\in\aut(\FC)$ extend $f_{\lambda}$.
After applying $f_{\lambda}$ to the set of formulae from (\ref{eq:6.1}), we obtain that
$$\pi(x)\,\cup\,\{\varphi_{\eta(0)}(x,b^{0}_l)\;|\;0\leqslant l\leqslant k_0\}\,\cup\,\{\varphi_{\eta(\alpha)}(x,b^{\alpha}_0)\;|\;0<\alpha<\lambda\}$$
is consistent which contradicts the maximality of $k_0$.

Now, assume that there are $(k_{\beta})_{\beta<\gamma}$ such that
\begin{IEEEeqnarray*}{rCl}
\pi(x) & \cup & \bigcup\limits_{\beta<\gamma} \{\varphi_{\eta(\beta)}(x,b^{\beta}_l)\;|\;0\leqslant l<k_{\beta}\} \\
& \cup & \bigcup\limits_{\beta<\gamma} \{\neg\varphi_{\eta(\beta)}(x,b^{\beta}_l)\;|\;l<0\text{ or }l\geqslant k_{\beta}\} \\
& \cup & \{\varphi_{\eta(\alpha)}(x,b^{\alpha}_0)\;|\;\gamma\leqslant\alpha<\lambda\}
\end{IEEEeqnarray*}
is consistent.
Let $k_{\gamma}<\omega$ be maximal such that
\begin{IEEEeqnarray*}{rCl}
\pi(x) & \cup & \bigcup\limits_{\beta<\gamma} \{\varphi_{\eta(\beta)}(x,b^{\beta}_l)\;|\;0\leqslant l<k_{\beta}\} \\
& \cup & \bigcup\limits_{\beta<\gamma} \{\neg\varphi_{\eta(\beta)}(x,b^{\beta}_l)\;|\;l<0\text{ or }l\geqslant k_{\beta}\} \\
& \cup & \{\varphi_{\eta(\gamma)}(x,b^{\gamma}_l)\;|\;0\leqslant l<k_{\gamma}\} \\
& \cup & \{\varphi_{\eta(\alpha)}(x,b^{\alpha}_0)\;|\;\gamma<\alpha<\lambda\}
\end{IEEEeqnarray*}
is consistent. Repeating the argument from the case of $\gamma=0$, we will obtain that also
\begin{IEEEeqnarray*}{rCl}
\pi(x) & \cup & \bigcup\limits_{\beta\leqslant\gamma} \{\varphi_{\eta(\beta)}(x,b^{\beta}_l)\;|\;0\leqslant l<k_{\beta}\} \\
& \cup & \bigcup\limits_{\beta\leqslant\gamma} \{\neg\varphi_{\eta(\beta)}(x,b^{\beta}_l)\;|\;l<0\text{ or }l\geqslant k_{\beta}\} \\
& \cup & \{\varphi_{\eta(\alpha)}(x,b^{\alpha}_0)\;|\;\gamma<\alpha<\lambda\}
\end{IEEEeqnarray*}
is consistent. \textit{Here ends the proof of Claim 1}.

Let 
$$m\models \pi(x)\,\cup\,\bigcup\limits_{\alpha<\lambda}\Big(\{\varphi_{\eta(\alpha)}(x,b^{\alpha}_l)\;|\;0\leqslant l<k_{\alpha}\}\,\cup\,\{\neg\varphi_{\eta(\alpha)}(x,b^{\alpha}_l)\;|\;l<0\text{ or }l\geqslant k_{\alpha}\}\Big),$$
$$c^{\alpha}_n:=(b^{\alpha}_{k_{\alpha}\cdot n}, b^{\alpha}_{k_{\alpha}\cdot n +1},\ldots,b^{\alpha}_{k_{\alpha}\cdot n +(k_{\alpha}-1)}),$$
$$\psi_{\alpha}(x,c^{\alpha}_n):=\bigwedge\limits_{0\leqslant j<k_{\alpha}}\varphi_{\eta(\alpha)}(x,b^{\alpha}_{k_{\alpha}\cdot n+j}),$$
where $\alpha<\lambda$ and $n\in \Zz$. Now, we are proceeding to show the first item from the thesis of the proposition, therefore fix arbitrary $\sigma:\lambda\to\Zz$.
\ 
\\
\\
\textbf{Claim 2:}
$$\exists m'\models \pi(x)\,\cup\,\bigcup\limits_{\alpha<\lambda}\{\psi_{\alpha}(x,c^{\alpha}_{\sigma(\alpha)}),\,\neg\psi_{\alpha}(x,c^{\alpha}_l)\;|\;l\neq\sigma(\alpha)\}.$$
\ 
\\
\textit{Proof of Claim 2:} First, induction on $\gamma<\lambda$:
\begin{IEEEeqnarray*}{rCl}
\exists m_{\gamma}  \models  \pi(x) & \cup & \bigcup\limits_{\alpha\leqslant\gamma} 
\{\psi_{\alpha}(x,c^{\alpha}_{\sigma(\alpha)}),\,\neg\psi_{\alpha}(x,c^{\alpha}_l)\;|\;l\neq\sigma(\alpha)\} \\
& \cup &\bigcup\limits_{\gamma<\alpha<\lambda} 
\{\psi_{\alpha}(x,c^{\alpha}_{0}),\,\neg\psi_{\alpha}(x,c^{\alpha}_l)\;|\;l\neq 0\}.
\end{IEEEeqnarray*}
We start with $m_{-1}:=m$. 

Now, assume that $m_{\gamma}$ is known and we need to find $m_{\gamma+1}$ for $\gamma+1<\lambda$.
As $b^{\gamma+1}_{\Zz}$ is $M^{\gamma+1}$-indiscernible, there exists 
$f_{\gamma+2}\in\aut(M^{\gamma+2}/M^{\gamma+1})$ such that $$f_{\gamma+2}(b^{\gamma+1}_{j})=b^{\gamma+1}_{k_{\gamma+1}\cdot \sigma(\gamma+1)+j}$$
for every $j\in\Zz$.
Note that also $$f_{\gamma+2}(c^{\gamma+1}_n)=c^{\gamma+1}_{n+\sigma(\gamma+1)}$$
for every $n\in\Zz$.

Set $\beta:=\gamma+1$. We recursively construct an increasing sequence $f_{\alpha}\in\aut(M^{\alpha}/M^{\beta})$,
where $\beta+1\leqslant\alpha\leqslant\lambda$.
The first automorphism $f_{\beta+1}=f_{\gamma+2}$ is already given.
Assume that we have everything up to $f_{\alpha}$ for some $\beta+1\leqslant\alpha\leqslant\lambda$
and we need to define $f_{\alpha+1}$.
Consider any extension $f_{\alpha}\subseteq f'_{\alpha}\in\aut(M^{\alpha+1}/M^{\beta})$.
As $b^{\alpha}_{\Zz}\models r_{\alpha}^{\otimes\Zz}|_{M^{\alpha}}$
and $r_{\alpha}$ is $M^{\alpha}/M^0$-invariant, we obtain that $f'_{\alpha}(b^{\alpha}_{\Zz})\models r_{\alpha}^{\otimes\Zz}|_{M^{\alpha}}$ and so $f'_{\alpha}(b^{\alpha}_{\Zz})\equiv_{M^{\alpha}}b^{\alpha}_{\Zz}$.
There exists $f''_{\alpha}\in\aut(M^{\alpha+1}/M^{\alpha})$ such that $f''_{\alpha}f'_{\alpha}(b^{\alpha}_{\Zz})=b^{\alpha}_{\Zz}$.
Set $f_{\alpha+1}:=f''_{\alpha}\circ f'_{\alpha}\in\aut(M^{\alpha+1}/M^{\beta})$.
If $\delta\leqslant\lambda$ is a limit ordinal then $f_{\delta}:=\bigcup\limits_{\beta+1\leqslant\alpha<\delta}f_{\alpha}\in\aut(M^{\delta}/M^{\beta})$.

We have that 
\begin{IEEEeqnarray*}{rCl}
f_{\lambda}(m_{\gamma})  \models  \pi(x) & \cup & \bigcup\limits_{\alpha\leqslant\gamma+1} 
\{\psi_{\alpha}(x,c^{\alpha}_{\sigma(\alpha)}),\,\neg\psi_{\alpha}(x,c^{\alpha}_l)\;|\;l\neq\sigma(\alpha)\} \\
& \cup &\bigcup\limits_{\gamma+1<\alpha<\lambda} 
\{\psi_{\alpha}(x,c^{\alpha}_{0}),\,\neg\psi_{\alpha}(x,c^{\alpha}_l)\;|\;l\neq 0\}.
\end{IEEEeqnarray*}
It is enough to set $m_{\gamma+1}:=f_{\lambda}(m_{\gamma})$.

If $\delta<\lambda$ is a limit ordinal and for every $\gamma<\delta$
\begin{IEEEeqnarray*}{rCl}
\pi(x) & \cup & \bigcup\limits_{\alpha\leqslant\gamma} 
\{\psi_{\alpha}(x,c^{\alpha}_{\sigma(\alpha)}),\,\neg\psi_{\alpha}(x,c^{\alpha}_l)\;|\;l\neq\sigma(\alpha)\} \\
& \cup &\bigcup\limits_{\gamma<\alpha<\lambda} 
\{\psi_{\alpha}(x,c^{\alpha}_{0}),\,\neg\psi_{\alpha}(x,c^{\alpha}_l)\;|\;l\neq 0\}.
\end{IEEEeqnarray*}
is consistent, then (with a little help of compactness) we have also that
\begin{IEEEeqnarray*}{rCl}
\pi(x) & \cup & \bigcup\limits_{\alpha<\delta} 
\{\psi_{\alpha}(x,c^{\alpha}_{\sigma(\alpha)}),\,\neg\psi_{\alpha}(x,c^{\alpha}_l)\;|\;l\neq\sigma(\alpha)\} \\
& \cup &\bigcup\limits_{\delta\leqslant\alpha<\lambda} 
\{\psi_{\alpha}(x,c^{\alpha}_{0}),\,\neg\psi_{\alpha}(x,c^{\alpha}_l)\;|\;l\neq 0\}.
\end{IEEEeqnarray*}
is consistent. Similarly as in the successor step, we can shift $c^{\delta}_{\Zz}$
keeping in place $\dom(\pi)$, $b^{<\delta}_{\Zz}$ and $b^{>\delta}_{\Zz}$
to obtain that 
\begin{IEEEeqnarray*}{rCl}
\pi(x) & \cup & \bigcup\limits_{\alpha\leqslant\delta} 
\{\psi_{\alpha}(x,c^{\alpha}_{\sigma(\alpha)}),\,\neg\psi_{\alpha}(x,c^{\alpha}_l)\;|\;l\neq\sigma(\alpha)\} \\
& \cup &\bigcup\limits_{\delta<\alpha<\lambda} 
\{\psi_{\alpha}(x,c^{\alpha}_{0}),\,\neg\psi_{\alpha}(x,c^{\alpha}_l)\;|\;l\neq 0\}.
\end{IEEEeqnarray*}
is consistent. Let $m_{\delta}$ be one of its realizations.

Assume that $(m_{\gamma})_{\gamma<\lambda}$ is the sequence obtained in the inductive process.
The thesis of Claim 2 follows by compactness. \textit{Here ends the proof of Claim 2}.

Claim 2 is the content of the first item from the thesis of the proposition.
For the second item of the thesis, we briefly argue shortly as follows.
Because of the condition (7) in Definition \ref{def:Qrank2}, for each $\alpha<\lambda$ there exists $l_{\alpha}<\omega$ such that $\{\varphi_{\eta(\alpha)}(x,b^{\alpha}_i)\;|\;i\in\Zz\}$ is $l_{\alpha}$-inconsistent. In particular for each $\alpha<\lambda$, we have that
$\{\psi_{\alpha}(x,c^{\alpha}_n)\;|\;n\in\Zz\}$ is $l_{\alpha}$-incosistent.
\end{proof}

\begin{fact}[Fact 2.11 in \cite{dpMIN}]\label{many} 
For any theory $T$ the following are equivalent.
\begin{enumerate}
\item $T$ is dp-minimal.

\item There is no sequence of  formulae $\psi_1(x,y), \dots, \psi_{2^n}(x,y)$ with $|x|=n$
and sequences $(a^j_i)_{i\in\omega}$ with $1 \leqslant j \leqslant 2^n$  
so that for any $\sigma: \{1, \dots, 2^n\} \to \omega$, the set
\[ \bigwedge_{1 \leqslant k \leqslant 2^n} \psi_k(x, a^k_{\sigma(k)}) \wedge \bigwedge_{1 \leqslant k \leqslant 2^n}
\bigwedge_{l\neq\sigma(k)}\neg\psi_k(x, a^k_l)\] 
is consistent.
\end{enumerate}
\end{fact}

\begin{corollary}
    If $T$ is dp-minimal then $\tilde{D}_Q(\pi(x))<2^{|x|}$.
\end{corollary}

\begin{proof}
By Proposition \ref{prop:tilde.rank} and Fact \ref{many}.
\end{proof}

\begin{question}
What other properties does $\tilde{D}_Q$ satisfyies in the dp-minimal context?
For example - do we have a counterpart of Lemma \ref{lemma:finite.character} in the dp-minimal context?
\end{question}

The following definitions may be found in \cite{ArtemNTP2}.

\begin{definition}
\begin{enumerate}
    \item An inp-pattern in $\pi(x)$ of depth $\lambda$ is $(l_{\alpha},\varphi_{\alpha}(x,y_{\alpha}),(b^{\alpha}_i)_{i<\omega})_{\alpha<\lambda}$ such that
    \begin{itemize}
        \item $\{\varphi_{\alpha}(x,b^{\alpha}_i)\;|\;i<\omega\}$ is $l_{\alpha}$-inconsistent for each $\alpha<\lambda$,
        \item $\pi(x)\,\cup\,\{\varphi_{\alpha}(x,b^{\alpha}_{\sigma(\alpha)})\;|\;\alpha<\lambda\}$ is consistent for each $\sigma:\lambda\to\omega$.
    \end{itemize}
    \item The burden of $\pi(x)$, denoted $\irk(\pi(x))$, is the the supremum of the depths of all inp-patterns in $\pi(x)$.
    
    \item 
    We let the dp-rank of $\pi(x)$, denoted $\drk(\pi(x))$, be the supremum of $\lambda$ for which there are $a\models\pi$ and mutually indiscernible over $C:=\dom(\pi)$
    sequences $(b^{\alpha}_i)_{i<\omega}$, where $\alpha<\lambda$,
    such that none of them is indiscernible over $aC$.
\end{enumerate}
\end{definition}

\begin{fact}\label{fact.NTP2}
\begin{enumerate}
    \item $\irk(\pi(x))\leqslant\drk(\pi(x))$ (Fact 3.8 in \cite{ArtemNTP2}).
    \item If $T$ is NIP then $\irk(\pi(x))=\drk(\pi(x))$ for every partial type $\pi(x)$ (Fact 3.8 in \cite{ArtemNTP2}).
    \item $T$ is NTP$_2$ iff $\irk(\pi(x))<\infty$ for every partial type $\pi(x)$ (Fact 
    3.2 in \cite{ArtemNTP2}). 
    \item $T$ is NIP iff $\drk(\pi(x))<\infty$ for every partial type $\pi(x)$ (Proposition 10 in \cite{Adler}). 
\end{enumerate}
\end{fact}

\begin{corollary}
    \begin{enumerate}
        \item For every $Q$ and partial type $\pi$, we have $$\tilde{D}_Q(\pi)\leqslant\irk(\pi)\leqslant\drk(\pi).$$
        
        \item If $T$ is NTP$_2$ then for every $Q$ and partial type $\pi$, we have
        $$\tilde{D}_Q(\pi)<\infty.$$
    \end{enumerate}
\end{corollary}

\begin{proof}
Follows from the definition of $\irk(\pi)$, Fact \ref{fact.NTP2} and Proposition \ref{prop:tilde.rank}.
\end{proof}

As we see above, $\tilde{Q}$-rank is finite in NSOP$_1$ and finite-dp-rankminimal theories. There is a natural question: what is the class of theories for which $\tilde{Q}$-rank is finite. We need to look for a class of theories which generalizes finite-dp-rankminimal and NSOP$_1$ theories at the same time. In their recent work, authors of \cite{NATP} provide a new dividing line in the stability hierarchy, called the \emph{antichain tree property} and show that theories without the antichain tree property (NATP) generalize NTP$_2$ and NSOP$_1$. 

\begin{question}
Is $\tilde{D}_Q$ finite in NATP?
Is $\tilde{D}_Q$ infinite in ATP?
\end{question}
\noindent 
A prospective approach to answer the second above question above, 
could be by considering the
$\tilde{Q}$-rank in the context of NSOP theories with SOP$_1$. 

Moreover, one could use finiteness of the $\tilde{Q}$-rank to indicate that a given theory
is placed among an extension of
\textit{the current known boundary of the combinatorially tame universe} \cite{Minh17}.
For example, the author of \cite{Minh17} notes that his theory ACFO of algebraically closed fields with multiplicative circular orders has TP$_2$ (Proposition 3.25 in \cite{Minh17}), but on the other hand his theory is a quite natural expansion of the theory of the algebraic closure of a finite field. 
For sure it is worth checking whether the $\tilde{Q}$-rank is finite (or ordinary-valued) in the case of ACFO.

\subsection{Example where rank is infinite}\label{ssec:exmample}
	Let $B$ be a Boolean algebra considered in language $L_{BA}=\{\wedge,\vee,^c, 0,1\}$ and $(P,\leqslant,0_P,1_P)$ a linearly ordered set with minimal element $0_P$ and maximal element $1_P$. We call a function $v:B\to P$ a valuation, if $v(x)=0_P$ iff $x=0$, $V(1)=1_P$, and $v(x\vee y)=\max(v(x),x(y))$ for all $x,y\in B$.
	Now let $L_{VBA}=L_{BA}\cup \{v,0_P,1_P\}$ be a two-sorted language on sorts $B$ and $P$, where $v$ is a symbol of a unary function from $B$ to $P$.  Let  $DVBA_0$ be the $L_{VBA}$-theory expressing that $(B,\vee,\wedge,^c,0,1)$ is an atomless Boolean algebra, $v:B\to P$ is a valuation, and for all $x\leqslant y$ in $B$ (here, ``$\leqslant$" indicates the canonical ordering in Boolean algebras) and any $p\in P$ such that $v(x)<p<v(y)$ there exists $z\in B$ with $x\leqslant z\leqslant y$  and $v(z)=p$ (density) and there are disjoint $y_1,y_2\leqslant y$ with $v(y_1)=v(y_2)=v(y)$ (no valuation-atoms).
	
	It is easy to see by compactness that $DVBA_0$ is consistent.
	By Example 4.33 from \cite{NATP}, $DVBA_0$ has ATP, so we intuitively expect that the $\tilde{Q}$-rank
	is infinite in $DVBA_0$. Let us argue on that.
	
	\begin{proposition}
		$DVBA_0$ is complete and admits QE in $L_{VBA}$.
	\end{proposition}
\begin{proof}
	Let $M,N\models DVBA_0$ and let $f:A\to B$ be an isomorphism with $A,B$ finite substructures of $M$ and $N$, respectively.  Let $a\in M$ be an element.  By a standard back-and-forth argument, it is enough to find $b\in B$ such that $f$ extends to an isomorphism of the substructures generated by $Aa$ and $Bb$. Let $(a_1,\dots,a_k)$ be all the atoms in $A$ and $(b_1,\dots,b_k)$ all the atoms in $B$. Let $P(A)=\{p_1,\dots,p_l\}$ and $P(B)=\{q_1,\dots,q_l\}$ be ascending enumerations in sort $P$.
	Assume $a\in B(M)$ (the argument in case $a\in P(M)$ is very simlar). 
    By the quantifier elimination in $(P,\leqslant, 0_P, 1_P)$, we can find $q'_1,\dots,q'_k,q''_1,\dots,q''_k\in P(N)$ such that  $(q_1,\dots,q_l,q'_1,\dots,q'_k,q''_1,\dots,q''_k)$ has the same quantifier-free type as $$(p_1,\dots,p_l,v(a\wedge a_1),\dots,v(a\wedge a_k),v(a^c\wedge a_1),\dots,v(a^c\wedge a_k)).$$
    By the density axiom, for every $i\leqslant k$ we can find $b'_i\leqslant b_i$ such that $v(b'_i)=q'_i$ and $v(b_i\wedge b_i^{\prime c})=q''_i$. Put $b:=\bigvee\limits_{i\leqslant k} b'_i$. Define $f'(a\wedge a_i)=b'_i$
    for all $i\leqslant k$ and 
    $$f'(v(a\wedge a_1),\dots,v(a\wedge a_k),v(a^c\wedge a_1),\dots,v(a^c\wedge a_k))=(q'_1,\dots,q'_k,q''_1,\dots,q''_k).$$
    Then $f\cup f'$ uniquely extends to an isomorphism of the substructures generated by $Aa$ and $Bb$ which sends $a$ to $b$, as required.
	\end{proof}

	Let $q(y)=\{0_P<v(y\wedge m)=v(y)<v(m')\;|\;m,m'\in B\setminus \{0\}\}$.
	We claim that $q(y)$ determines a complete global type (note it will obviously be $\emptyset$-invariant). Clearly $q(y)$ determines a complete type of $y$ in the Boolean algebra $B$. Hence, by the quantifier elimination,
	it is enough to prove that for any $x,y\models q$ and any $a,b,c,d\in B$ we have 
	$v((x\wedge a )\vee (x^c\wedge b))<v((x\wedge c )\vee (x^c\wedge d))$ if and only if $v((y\wedge a )\vee (y^c\wedge b))<v((y\wedge c )\vee (y\wedge d))$, and likewise for equality in place of inequality. Note that if $b\neq  0$ then $v(x^c\wedge b)=v(b)>v(x\wedge a )$, so $v((x\wedge a )\vee (x^c\wedge b))=v(b)$. On the other hand, if $b=0$ then $v((x\wedge a )\vee (x^c\wedge b))=v(x\wedge a)=v(x)$. Hence, the equivalences easily follow by inspection.
	
	Now, for any model $M$ define $r_M(y)=\{0_P<v(y\wedge m)=v(y)<v(m')\;|\;m,m'\in B, v(m)\in conv(v[B(M)\setminus \{0\}])\}\cup\{y\wedge c=0\;|\;v(c)<v[B(M)\setminus \{0\}]\}$, where $conv$ denotes the convex hull operation in $(P,\leqslant)$. Note $q|_M\subseteq r_M$. Again, by quantifier elimination $r_M(y)$ determines a complete global type. Moreover, as $(v[B(M)\setminus 0])$ and its convex hull are invariant under automorphisms preserving $M$ setwise, we get that $r_M(y)$ is $M/M_0$-invariant for any $M_0\prec M$.
	
	Let $\varphi(x,y)=(x\neq 0\wedge x\leqslant y)$ and $Q=(\varphi(x,y),q(y))$. We claim that $\tilde{D}_Q(x=x)=\infty$ for $x$ a variable of the sort $B$. 
	Fix any ordinal $\lambda$. Choose  $(M^\alpha, b^\alpha)_{\alpha<\lambda}$ 
	with $(M^\alpha)_{\alpha<\lambda}$ continuous so that 
	$b^\alpha\models q|_{M^\alpha}$ and $b^{\alpha}M^{\alpha}\subseteq M^{\alpha+1}$ for every $\alpha<\lambda$,
	and each $M^{\alpha+1}$ is $|M^{\alpha}|^+$-saturated and strongly $|M^{\alpha}|^+$-homogeneous.
	Then $r|_{M^\alpha}$ witnesses that $\varphi(x,b^\alpha)$ Kim-divides (with respect to a semi-global type) over $M^\alpha$ for every $\alpha<\lambda$. 
	On the other hand, $\{\varphi(x,b^\alpha)\;|\;\alpha<\lambda\}$ is consistent by compactness and the choice of $q$. This shows that $\tilde{D}_Q(x=x)\geqslant \lambda$, so $\tilde{D}_Q(x=x)=\infty$.

\bibliographystyle{plain}
\bibliography{1nacfa2}

\begin{thebibliography}{10}

\bibitem{Adler}
Hans Adler.
\newblock Strong theories, burden, and weight.
\newblock available at
  \url{http://www.logic.univie.ac.at/~adler/docs/strong.pdf} .

\bibitem{HansThesis}
Hans Adler.
\newblock {\em Explanation of {I}ndependence}.
\newblock PhD thesis, Albert-Ludwigs-Universit\"{a}t Freiburg im Breisgau,
  2005.

\bibitem{NATP}
JinHoo Ahn, Joonhee Kim, and Junguk Lee.
\newblock On the antichain tree property.
\newblock available on https://arxiv.org/abs/2106.03779.

\bibitem{casasimpl}
Enrique Casanovas.
\newblock {\em Simple theories and hyperimaginaries}.
\newblock Lecture Notes in Logic. Cambridge University Press, 2011.

\bibitem{zoeIP}
Zoé Chatzidakis.
\newblock Model theory of profinite groups having the iwasawa property.
\newblock {\em Illinois J. Math.}, 42(1):70--96, 03 1998.

\bibitem{zoe2002}
Zo{\'e} Chatzidakis.
\newblock Properties of forking in $\omega$-free pseudo-algebraically closed
  fields.
\newblock {\em Journal of Symbolic Logic}, 67(3):957–996, 2002.

\bibitem{ChaPil}
Zo{\'e} Chatzidakis and Anand Pillay.
\newblock Generic structures and simple theories.
\newblock {\em Annals of Pure and Applied Logic}, 95(1-3):71--92, 1998.

\bibitem{cherlindriesmacintyre}
G.~Cherlin, L.~van~den Dries, and A.~Macintyre.
\newblock The elementary theory of regularly closed fields.
\newblock Available on
  \verb"http://sites.math.rutgers.edu/~cherlin/Preprint/CDM2.pdf".

\bibitem{ArtemNTP2}
Artem Chernikov.
\newblock Theories without the tree property of the second kind.
\newblock {\em Annals of Pure and Applied Logic}, 165(2):695--723, 2014.

\bibitem{CKR}
Artem Chernikov, Kim Byungham, and Nicholas Ramsey.
\newblock Transitivity, lowness, and ranks in {N}{S}{O}{P}$_1$ theories.
\newblock available on \verb"https://arxiv.org/pdf/2006.10486.pdf".

\bibitem{ArtemNick}
Artem Chernikov and Nicholas Ramsey.
\newblock On model-theoretic tree properties.
\newblock {\em Journal of Mathematical Logic}, 16(02):1650009, 2016.

\bibitem{conant_kruckman_2019}
Gabriel Conant and Alex Kruckman.
\newblock Independence in generic incidence structures.
\newblock {\em The Journal of Symbolic Logic}, 84(2):750–780, 2019.

\bibitem{dElbee21}
Christian D'Elb\'{e}e.
\newblock Forking, imaginaries and other features of acfg.
\newblock {\em The Journal of Symbolic Logic}, page 1–34, 2021.

\bibitem{bilinear}
Jan Dobrowolski.
\newblock Sets, groups, and fields definable in vector spaces with a bilinear
  form.
\newblock available on \verb"https://arxiv.org/abs/2004.07238".

\bibitem{DKR}
Jan Dobrowolski, Byunghan Kim, and Nicholas Ramsey.
\newblock Independence over arbitrary sets in nsop$_1$ theories.
\newblock available on \verb"https://arxiv.org/pdf/1909.08368.pdf".

\bibitem{dpMIN}
Alfred Dolich, John Goodrick, and David Lippel.
\newblock {Dp-Minimality: Basic Facts and Examples}.
\newblock {\em Notre Dame Journal of Formal Logic}, 52(3):267 -- 288, 2011.

\bibitem{sheDza}
Mirna D\v{z}amonja and Saharon Shelah.
\newblock On $\triangleleft^*$—maximality.
\newblock {\em Annals of Pure and Applied Logic}, 125(1):119--158, 2004.

\bibitem{Granger}
Nicholas Granger.
\newblock Stability, simplicity, and the model theory of bilinear forms.
\newblock PhD thesis, University of Manchester, 1999.

\bibitem{HL1}
Daniel~Max Hoffmann and Junguk Lee.
\newblock Co-theory of sorted profinite groups for {P}{A}{C} structures.
\newblock Submitted, available on \verb"https://arxiv.org/abs/1905.09748".

\bibitem{kaplanramsey2017}
Itay Kaplan and Nicholas Ramsey.
\newblock On {K}im-independence.
\newblock {\em Journal of the European Mathematical Society}, 22:1423 -- 1474,
  2020.

\bibitem{kaplanramsey2021}
Itay Kaplan and Nicholas Ramsey.
\newblock Transitivity of kim-independence.
\newblock {\em Advances in Mathematics}, 379:107573, 2021.

\bibitem{KRS}
Itay Kaplan, Nicholas Ramsey, and Saharon Shelah.
\newblock Local character of kim-independence.
\newblock {\em Proc. Amer. Math. Soc.}, 147:1719--1732, 2019.

\bibitem{KimTalk}
Byunghan Kim.
\newblock Ntp$_1$ theories. slides.
\newblock Stability {T}heoretic {M}ethods in {U}nstable {T}heories,
  {B}{I}{R}{S}, 2009.

\bibitem{KimPi}
Byunghan Kim and Anand Pillay.
\newblock Simple theories.
\newblock {\em Annals of Pure and Applied Logic}, 88(2–3):149 -- 164, 1997.
\newblock Joint AILA-KGS Model Theory Meeting.

\bibitem{simpleGenerics}
Anand Pillay.
\newblock Definability and definable groups in simple theories.
\newblock {\em The Journal of Symbolic Logic}, 63(3):788--796, 1998.

\bibitem{Anand98}
Anand Pillay.
\newblock Definability and definable groups in simple theories.
\newblock {\em Journal of Symbolic Logic}, 63(3):788–796, 1998.

\bibitem{tentzieg}
Katrin Tent and Martin Ziegler.
\newblock {\em A Course in Model Theory}.
\newblock Lecture Notes in Logic. Cambridge University Press, 2012.

\bibitem{Minh17}
Chieu-Minh Tran.
\newblock Tame structures via sums over finite fields.
\newblock available on \verb"https://arxiv.org/pdf/1704.03853.pdf".

\end{thebibliography}

\end{document}